\documentclass[review,hidelinks,onefignum,onetabnum]{siamart220329}
\nolinenumbers
\usepackage{lipsum}
\usepackage{amsfonts}
\usepackage{graphicx}
\usepackage{epstopdf}
\usepackage{algorithmic}
\usepackage{cleveref}
\usepackage{tikz-cd}
\usepackage{amsmath}
\usepackage{subcaption}
\usepackage{mathtools} 
\usepackage{comment}
\usepackage{todonotes}
\usepackage{bm}
\usepackage[sort]{cite}

\let\ker\undefined
\DeclareMathOperator{\ker}{ker}

\renewcommand\div{\operatorname{div}}

\newcommand\curl{\operatorname{curl}}
\newcommand\grad{\operatorname{grad}}

\ifpdf
  \DeclareGraphicsExtensions{.eps,.pdf,.png,.jpg}
\else
  \DeclareGraphicsExtensions{.eps}
\fi

\newsiamremark{remark}{Remark}
\newsiamremark{hypothesis}{Hypothesis}
\newsiamthm{problem}{Problem}
\crefname{hypothesis}{Hypothesis}{Hypotheses}

\newsiamthm{claim}{Claim}

\headers{Discretization of magneto-frictional equations}{M.~He, P.~E.~Farrell, K.~Hu, B.~D.~Andrews}

\title{Helicity-preserving finite element discretization for magnetic relaxation
\thanks{Submitted to the editors DATE.
\funding{This work was funded by the Engineering and Physical Sciences Research Council
[grant numbers EP/R029423/1 and EP/W026163/1], the EPSRC Energy Programme [grant number
EP/W006839/1], a CASE award from the UK Atomic Energy Authority, the Donatio Universitatis
Carolinae Chair ``Mathematical modelling of multicomponent systems'', the European Research Council (ERC Starting Grant, project 101164551 GeoFEM), and by a Royal Society University Research Fellowship (URF$\backslash$R1$\backslash$221398).} 
This work used the ARCHER2 UK National Supercomputing Service.
For the purpose of open access, the authors have applied a CC BY public copyright licence to any author accepted manuscript (AAM) arising from this submission. No new data were generated or analysed during this work.}
}

\author{Mingdong He\thanks{Mathematical Institute, University of Oxford, UK
  (\email{mingdong.he@maths.ox.ac.uk}).}
\and Patrick E.~Farrell\thanks{Mathematical Institute, University of Oxford, UK
     and
     Mathematical Institute, Faculty of Mathematics and Physics, Charles University, Czechia
  (\email{patrick.farrell@maths.ox.ac.uk}).}
\and Kaibo Hu\thanks{School of Mathematical Sciences, University of Edinburgh, UK
  (\email{kaibo.hu@ed.ac.uk}).}
\and Boris D.~Andrews\thanks{Mathematical Institute, University of Oxford, UK
  (\email{boris.andrews@maths.ox.ac.uk}).}
}

\usepackage{amsopn}

\newcommand{\bdainline}[1]{\todo[color=orange!30,inline]{{\bf BDA:} #1}}

\ifpdf
\hypersetup{
  pdftitle={Topology-preserving discretization for the magneto-frictional equations arising in the Parker conjecture},
  pdfauthor={M.~He, P.~E.~Farrell, K.~Hu, B.~D.~Andrews}
}
\fi

\newcommand{\dx}{\,\mathrm{d}x}

\begin{document}

\maketitle

% REQUIRED
\begin{abstract}
The Parker conjecture, which explores whether magnetic fields in perfectly conducting plasmas can develop tangential discontinuities during magnetic relaxation, remains an open question in astrophysics.
Helicity conservation provides a topological barrier during relaxation, preventing topologically nontrivial initial data relaxing to trivial solutions;
preserving this mechanism discretely over long time periods is therefore crucial for numerical simulation.
This work presents an energy- and helicity-preserving finite element discretization for the magneto-frictional system for investigating the Parker conjecture. The algorithm preserves a discrete version of the topological barrier and a discrete Arnold inequality. We also propose extensions of the notion of helicity and the Arnold inequality to certain kinds of topologically nontrivial domains. Numerical experiments demonstrate that helicity preservation is crucial in obtaining physically meaningful simulations of magnetic relaxation, providing an example where structure-preserving schemes are necessary.
\begin{comment}
Beyond the magneto-frictional system itself, our work highlights why structure-preserving schemes are necessary. In many problems, they offer huge quantitative improvements, but standard discretizations still provide reasonable results. Whereas in this work, the improvement is qualitative—without helicity preservation, the solutions collapse to unphysical trivial states. 
\end{comment}
\end{abstract}

% REQUIRED

% REQUIRED

\section{Introduction}
In 1972, Eugene N.~Parker made a conjecture about the behaviour of magnetically ideal plasmas~\cite{parkerTopologicalDissipationSmallScale1972}.
Although several different versions and statements have been discussed in the literature \cite{pontinParkerProblemExistence2020}, the essential claim of the Parker conjecture is as follows:
\emph{For almost all possible flows, the magnetic field develops tangential discontinuities (current sheets) during ideal magnetic relaxation to a force-free equilibrium.} This is also known as the force-free version of the Parker conjecture, and it has many important consequences in solar physics, including explaining the mechanism of coronal heating.
For a comprehensive literature review, see \cite{pontinParkerProblemExistence2020}.

We focus in this work on the case energy is dissipated by fluid viscosity (the fluid Reynolds number $R_{e}<\infty$), while the magnetic part is ideal (the magnetic Reynolds number $R_{m}=\infty$).
%\bdainline{I've removed this paragraph. I don't know if it contributes anything? It adds some notation, but it doesn't seem like it's ever used again?}
%\bdainline{This jump between paragraphs is also a little jarring to me, maybe? We state the Parker conjecture, then start talking about invariants, and it's not obvious how the two are related. Maybe something saying, like, ``we are interested in numerical investigations into the Parker conjecture [...] this necessitates schemes that preserve the long-term behaviour of magnetic relaxation [...] bla bla bla''.}

For general MHD systems, standard energy estimate indicate that the total energy of the fluid and the magnetic field is non-increasing.
Over a bounded, contractible, Lipschitz domain $\Omega \subset \mathbb{R}^3$, the (magnetic) helicity is defined
\begin{equation}\label{eq:helicity}
    \mathcal{H} \coloneqq \int_{\Omega} \bm A \cdot \bm B \dx.
\end{equation}
Here $\bm A$ is any potential for the magnetic field $\bm B$ satisfying $\bm B = \curl \bm A$ and $\bm A \times \bm n = 0$ on the boundary $\partial \Omega$, where $\bm n$ is the (outward-facing) unit normal.
In the magnetically ideal limit, this is an invariant under relaxation:
\begin{equation}\label{eqn:helicity-preserved}
    \frac{d}{dt}\mathcal{H} = 0.
\end{equation}
The helicity provides a lower bound for the magnetic energy
\begin{equation}\label{eqn:arnold-inequality}
    |\mathcal{H}|\leq C^{-1}\int_{\Omega} \bm B\cdot \bm B \dx,
\end{equation}
for some constant $C > 0$.
This is known as the Arnold inequality \cite{arnold1974asymptotic}; for completeness, we include a proof in \Cref{sec:proof-arnold-ineq}.
A significant consequence of the structure \eqref{eqn:arnold-inequality} is that, despite dissipation in the energy, initial data with nonzero $\mathcal{H}$ cannot relax to zero, as the total energy cannot decay below a certain multiple of $|\mathcal{H}|$.
As helicity quantifies the knottedness of divergence-free fields \cite{ArnoldTopologicalMethodsHydrodynamics2021,moffatt1981some}, this reflects a topological barrier exhibited by knotted magnetic fields.
This property is crucial for the mathematical and physical behaviour of magnetic relaxation.

In numerical computation, however, the magnetic helicity $\mathcal{H}$ is typically not conserved (or even potentially ill-defined if the discrete magnetic field is not divergence-free).
Thus, the energy of the numerical magnetic field may decay to zero, leading to nonphysical solutions.
Enforcing a discrete analogue of this topological barrier is therefore crucial for meaningful long-term simulations.
A further challenge for the numerical investigation of the Parker conjecture is the resolution of tangential discontinuities in $\bm B$ \cite{pontinParkerProblemExistence2020}.
Numerical representations of the magnetic field must allow such discontinuities.

In this paper we consider the magneto-frictional system, a simplified version of the full time-dependent MHD equations \cite{Chodura3DcodeMHD1981,yeates2022limitations},
\begin{subequations}\label{eqn:magneto-frictional-equations}
\begin{align}
    \label{eqn:magnetic-advection}\partial_t \bm{B} + \curl\bm{E}&=\bm{0},\\
    \label{eqn:electric-field}\bm{E}+ \bm{u}\times\bm{B} &= \bm{0}, \\
    \bm{j}& = \curl\bm{B}, \\
    \label{eqn:velocity}\bm{u} &= \tau \bm{j}\times\bm{B},
\end{align}
\end{subequations}
with initial data $\bm{B}|_{t=0} = \bm{B}_0$ satisfying the magnetic Gauss law $\div\bm B_0 = 0$;
as a consequence of the magnetic advection equation \eqref{eqn:magnetic-advection}, the magnetic Gauss law holds for $\bm B$ at all times $t > 0$.
The parameter $\tau > 0$ is a coupling parameter.

Instead of coupling the velocity field $\bm{u}$ via the Navier--Stokes equations, the magneto-frictional equations are closed by assuming a special form of $\bm{u}$ \eqref{eqn:velocity}.
This guarantees the magnetic energy $\mathcal{E}(t) \coloneqq \int |\bm B|^2\dx$ is non-increasing, with rate
\begin{equation}\label{eqn:energy-equality}
    \partial_t \mathcal{E} = - 2\tau \int_{\Omega}|\bm j\times \bm B|^2 \dx,
\end{equation}
thus avoiding the exchange of kinetic energy and magnetic energy permitted in the original MHD system.
The system furthermore retains the conservation of magnetic helicity $\mathcal{H}$, preventing the decay to zero during relaxation via the Arnold inequality \cref{eqn:arnold-inequality} for initial data with non-zero $\mathcal{H}$.
At a stationary state ($\partial_t {\bm B}=0$), the dissipation of energy \cref{eqn:energy-equality} implies the Lorentz force $\bm j\times \bm{B}$ must vanish.
Should such a solution exist, it is referred to as a nonlinear force-free field or Beltrami field\cite{pontinParkerProblemExistence2020}, coinciding with an equilibrium of the full MHD system when $\bm u$ is given by the coupling with the Navier--Stokes equations.
% It also solves the full MHD equation, where $\bm u$ is given by the coupling with the Navier--Stokes equations.
The existence of tangential discontinuities of the stationary solutions of \cref{eqn:magneto-frictional-equations} is therefore equivalent to the {force-free version of} the Parker conjecture \cite[Def.~2]{pontinParkerProblemExistence2020}. 
%\emph{For almost all possible boundary flows, the magnetic field develops tangential discontinuities during relaxation to a force-free equilibrium.}
We consider the following boundary conditions
\begin{subequations}
\begin{equation}
    \bm{B}\cdot\bm{n} = 0,  \quad \bm j \times \bm n = \bm{0}, \quad \text{on } \partial \Omega.
\end{equation}
These imply
\begin{equation}
   \bm{u}\cdot\bm{n} = 0, \quad \bm{E} \times \bm{n} = \bm{0}, \quad \text{on } \partial\Omega.
\end{equation}
\end{subequations}

The analysis of the magneto-frictional system \eqref{eqn:magneto-frictional-equations} is of broad interest. Whether the solution $\bm B$ is able to develop finite-time singularities remains open; for recent progress in this direction, see \cite{brenierTopologyPreservingDiffusionDivergenceFree2014,beekie2022moffatt,enciso2025obstructions}.

The numerical challenges described above (i.e.~both conserving the magnetic helicity, and resolving tangential discontinuities) are pivotal when considering the magneto-frictional system \cite{pontinParkerProblemExistence2020}. 
There is a rich literature on numerical investigation of magnetic relaxation.
Most existing Eulerian methods introduce artificial magnetic reconnection;
the topological structure is therefore lost, leading to non-physical solutions \cite{zweibel1987formation}.
%\bda{I mean, we have some magnetic reconnection, no? We don't preserve the topology exactly. I don't know if this is a fair comparison.}

Regarding Lagrangian discretizations, Craig \& Sneyd \cite{craig1986dynamic} developed a finite-difference relaxation method to preserve the topology of the magnetic fields;
this has been used extensively to study the Parker conjecture \cite{longbottom1998magnetic, craig2005parker, wilmot2009magneticquasi, wilmot2009magneticparallel, craig2014current}.
A mimetic Lagrangian {method} was proposed in \cite{candelaresiMimeticMethodsLagrangian2014a} to ensure charge conservation.
A variational integrator for ideal MHD has been developed by \cite{zhou2014variational}, which is both symplectic and momentum preserving;
this method has been applied to examine the singularity of the nonlinear force-free field \cite{zhou2016formation, zhou2017constructing}.
However, Lagrangian methods suffer from strong mesh distortions, preventing one from obtaining more conclusive results.

%\bdainline{Why does mesh distortion make the results any less ``conclusive''?}
In this paper, we design a structure-preserving Eulerian finite-element discretization for the magneto-frictional system.
Our method relies on the finite element exterior calculus (FEEC) \cite{ArnoldFiniteElementExterior2018,arnoldFiniteElementExterior2006,huHelicityconservativeFiniteElement2021} and a recently proposed framework for enforcing conservation laws and dissipation inequalities \cite{andrews2024enforcing}.
We believe this to be the first helicity-preserving Eulerian finite element method for this problem, which is highly desirable in computational solar physics.

The fields in our scheme are discretized via finite element de~Rham complexes.
In particular, the magnetic field $\bm B$ is discretized with the Raviart--Thomas \cite{raviart2006mixed} or Brezzi--Douglas--Marini \cite{brezzi1985two} finite element spaces, which (crucially) allow tangential discontinuities.
The method preserves both the dissipation of energy and conservation of helicity, thus ensuring the topological barrier provided by the Arnold inequality \eqref{eqn:arnold-inequality} is inherited on the discrete level (see \eqref{eqn:discrete-arnold-inequality} below) preventing some artificial magnetic reconnection.
The method is valid for unstructured meshes on general domains, avoids mesh distortions, and can be of arbitrary order in both space and time.

The Parker conjecture is often studied in settings with periodic boundary conditions (topologically equivalent to a torus) yielding a domain with nontrivial topology.
However, the magnetic potential $\bm A$, present in the definition of the helicity \eqref{eq:helicity}, is ill-defined in the presence of such topological irregularities.
Through a careful handling of certain harmonic forms, we further propose a conserved generalized helicity that is well-defined over certain topologically nontrivial domains;
the generalized helicity similarly demonstrates a (generalized) Arnold inequality \eqref{eqn:arnold-inequality}.
Over such domains, our discretisation again preserves this generalized helicity, thus preserving the topological barrier.

Structure-preserving (or compatible) discretizations have been intensively studied in several contexts, including in geometric numerical integration \cite{hairer2006geometric} and FEEC \cite{arnoldFiniteElementExterior2010}.
Various examples have demonstrated the crucial role of preserving certain structures.
For instance, long-term simulations for the Kepler problem with non-structure-preserving integrators may lead to accumulated errors in the orbit \cite[Fig.~2.2]{hairer2006geometric}.
For Maxwell source and eigenvalue problems, finite element methods that do not respect the cohomological structures may lead to spurious solutions \cite{arnoldFiniteElementExterior2010}.

%\bdainline{I don't know if we need any of the above 4 (commented out) sentences. Having just a few tangentially related examples makes structure preservation seem far \emph{less} iportant than it is, in my opinion. I've commented them out (for now) and propose the following instead:}

De~Rham complex--based methods have been widely discussed for constructing structure-preserving discretizations in MHD \cite{huHelicityconservativeFiniteElement2021,huStableFiniteElement2017,huStructurepreservingFiniteElement2018,LaakmannStructurepreservinghelicityconservingfinite2023,MaRobustpreconditionersincompressible2016,gawlikFiniteElementMethod2022}.
Most, however, are developed primarily from analytic perspectives. The fundamental question of the qualitative importance of helicity-preservation is less often addressed, with numerical examples of their physical significance remaining rare. A key contribution of this work is in demonstrating that, to obtain meaningful numerical solutions for the magneto-frictional equations \eqref{eqn:magneto-frictional-equations}, the discrete conservation of helicity is crucial.

The remainder of this paper proceeds as follows.
In \Cref{sec:Preliminaries}, we introduce certain preliminaries.
In \Cref{sec:Structure-preserving discretization}, we propose our structure-preserving scheme;
through a generalized notion of helicity, we extend the Arnold inequality to domains with nontrivial topology.
In \Cref{sec:Numerical-experiment}, we present numerical results, and compare our proposed method to other discretizations to explore the physical significance of helicity preservation.
In \Cref{sec:Conclusion}, we conclude with a general outlook on the numerical simulation of the Parker conjecture.

\section{Notation and preliminaries}\label{sec:Preliminaries}
Let $\Omega\subset \mathbb{R}^3$ be a bounded, Lipschitz domain;
if not otherwise specified, we assume that $\Omega$ is contractible.
Again, let $\bm n$ denote the outward-pointing unit normal vector on $\partial\Omega$.
We use $\|\cdot\|$ and $(\cdot, \cdot)$ to denote the $L^2(\Omega)$ norm and inner product respectively, allowing $L^2$ to denote both the scalar- and vector-valued spaces.

Define the following Hilbert spaces:
\begin{subequations}
\begin{align}
    H^1       &\coloneqq  \{v\in L^2(\Omega) : \nabla v \in L^2(\Omega)\},  \\
    H(\curl)  &\coloneqq  \{\bm v\in L^2(\Omega) : \curl\bm v \in L^2(\Omega)\},  \\
    H(\div)   &\coloneqq  \{\bm v\in L^2(\Omega) : \div\bm v \in L^2(\Omega)\}.
\end{align}
\end{subequations}
Restricting on the boundary $\partial\Omega$, we define the following subspaces:
\begin{subequations}
\begin{align}
    H^1_0       &\coloneqq  \{v\in H^1 : v = 0 \text{ on } \partial\Omega\},  \\
    H_0(\curl)  &\coloneqq  \{\bm v\in H(\curl) : \bm v\times\bm n = \bm 0 \text{ on } \partial\Omega\},  \\
    H_0(\div)  &\coloneqq  \{\bm v\in H(\div) : \bm v\cdot\bm n = 0 \text{ on } \partial\Omega\}.
\end{align}
\end{subequations}
Moreover, we define
\begin{equation}
    L^2_0(\Omega)  \coloneqq  \{u\in L^2(\Omega) : \int_{\Omega}u=0\}.
\end{equation}
The 3D de~Rham complex with homogeneous boundary conditions (BCs) reads:
\begin{subequations}
\begin{equation}
    \begin{tikzcd}
    0\arrow{r}&H_0^1 \arrow[r, "\grad"] &
    H_0(\curl)\arrow[r,"\curl"] &
    H_0(\div) \arrow[r,"\div"] &
    L_0^2 \arrow{r} & 0.
    \end{tikzcd}
    \label{eq:deRhamcomplex}
\end{equation}
This complex is exact on contractible domains.
We will use finite-element subcomplexes of \eqref{eq:deRhamcomplex} for discretization;
several such subcomplexes are well-known, consisting of N\'ed\'elec \cite{nedelec1-0}, Raviart--Thomas\cite{raviart2006mixed}, and Brezzi--Douglas--Marini elements \cite{brezzi1985two}, each extending to arbitrary spatial order.
Adopting the notation of \cite{ArnoldFiniteElementExterior2018} we denote such a subcomplex by
\begin{equation}
    \begin{tikzcd}
0\arrow{r}& H_{0}^{1,h}\arrow[r, "\grad"] &
    H_0^h(\curl)\arrow[r,"\curl"] &
    H_0^h(\div)\arrow[r,"\div"] &
    L_0^{2,h} \arrow{r} & 0.
    \end{tikzcd}
    \label{eq:subcomplex}
\end{equation}
\end{subequations}
We require that \eqref{eq:subcomplex} is exact on contractible domains.

\section{Structure-preserving discretization}\label{sec:Structure-preserving discretization}

\begin{comment}
\bdainline{Following Patrick's comments, I've written out an idea below about how I might briefly motivate the formulation of the scheme/why we introduce the variables we do. Hopefully this exposition also clarifies the questions about the test and solution space mismatch.}
\end{comment}
\begin{comment}
For the magnetic helicity to be well-defined on the discrete level, our scheme must preserve the magnetic Gauss law ($\div \bm B=0$) at least up to solver tolerances and machine precision;
we therefore discretize the magnetic field $\bm B_h$ in the discrete $H_0^h(\div)$ space \eqref{eq:subcomplex}.
Since $\partial_t\bm B_h$ and $\curl \bm E_h$ are connected in \eqref{eqn:magnetic-advection}, we choose to discretize $\bm E_h$ in the $H_0^h(\curl)$ space, as in \cite{huStableFiniteElement2017}.
To preserve the energy dissipation and helicity conservation structures, the framework \cite{andrews2024enforcing} introduces auxiliary variables representing the current ($\bm j_h = \mathbb{Q}_c\curl\bm B_h$) and magnetic field ($\bm H_h = \mathbb{Q}_c\bm B_h$) respectively.
\end{comment}
\begin{comment}
\bdainline{This footnote is pretty long. If we want to go into less detail, we could just say something like ``\emph{The choice of spaces and variables to achieve helicity- and energy-preservation can be systematically derived from the framework recently proposed by Andrews \& Farrell \cite{AndrewsHighorderconservativeaccurately2024}''.}}
\end{comment}

For the magnetic helicity to be well-defined on the discrete level, our scheme must preserve the magnetic Gauss law ($\div \bm B=0$) at least up to solver tolerances and machine precision;
we therefore discretize the magnetic field $\bm B_h$ in the $H_0^h(\div)$ space  \eqref{eq:subcomplex}.
Since $\curl H_0^h(\curl)\subset H_0^h(\div)$, the advection equation \eqref{eqn:magnetic-advection} then suggests we discretize $\bm E_h$ in the $H_0^h(\curl)$ space.
As shown in \Cref{thm:sp-properties} below, this is sufficient to ensure $\partial_t\bm B_h + \curl \bm E_h = \bm 0$ exactly, additionally guaranteeing that if $\div \bm B_0 = 0$ then the magnetic Gauss law holds pointwise on $\bm B_h$ for every $t$.

To replicate the right physics, we must preserve both the helicity conservation \eqref{eqn:helicity-preserved} and energy dissipation \eqref{eqn:energy-equality} laws.
The general idea proposed in \cite{andrews2024enforcing} for designing numerical schemes that replicate conservation and dissipation properties is to rewrite these laws as time integrals involving certain associated test functions;
discrete approximations of the associated test functions are then introduced as auxiliary variables in the scheme.
In this context, the two auxiliary variables introduced for helicity conservation and energy dissipation are an auxiliary current $\bm j_h$ and magnetic field $\bm H_h$, $L^2$ projections of $\curl \bm B_h$ and $\bm B_h$ onto the space $H_0^h(\curl)$ respectively.\footnote{We may equivalently write $\bm j_h = \curl_h \bm B_h$, where $\curl_h$ is the $L^2$ adjoint operator of $\curl$ in the discrete complex \eqref{eq:subcomplex}.}
These introduced auxiliary variables resemble those proposed for similar MHD systems in \cite{huHelicityconservativeFiniteElement2021,LaakmannStructurepreservinghelicityconservingfinite2023} with extension to higher order in time.

We therefore propose the following semi-discrete structure-preserving discretization of the magneto-frictional equations \eqref{eqn:magneto-frictional-equations}.
\begin{comment}
\footnote{
    To derive \eqref{eq:structure-preserving-scheme} using the framework of \cite{AndrewsHighorderconservativeaccurately2024} it is convenient to first reparametrize the system in the magnetic potential $\bm A$ and electric potential $\phi$,
    \begin{equation}\label{eqn:magneto-frictional-equations-compact}
        \dot{\bm A} + \grad\phi  =  \tau\curl[(\curl^2\bm A \times \curl\bm A) \times \curl\bm A],  \quad
        \div \bm A  =  0.
    \end{equation}
    The divergence-free condition and $\phi$ may both be eliminated in a weak formulation of \eqref{eqn:magneto-frictional-equations-compact} by solving and testing (in $L^2$) over a discretely divergence-free finite element space (as done in the Navier--Stokes example in \cite[Chap.~2]{AndrewsHighorderconservativeaccurately2024}). Seeking $\bm A \in H_0^h(\curl)$ and $\phi \in H_0^{1,h}$ then allows the use of the subcomplex \eqref{eq:subcomplex} to reparametrize in $\bm B = \curl \bm A \in H_0^h(\div)$ and $\bm E = - \dot{\bm A} - \grad\phi \in H_0^h(\curl)$.
}.
\end{comment}
\begin{problem}[Structure-preserving semi-discrete scheme]
    Given $\bm B_0 \in H_0^h(\div)$, find $\bm B_h \in C^1(\mathbb{R}_+; H_0^h(\div))$ satisfying the initial conditions (ICs) $\bm B|_{t=0} = \bm B_0$ and $(\bm E_h, \bm j_h, \bm H_h) \in C^0(\mathbb{R}_+; H_0^h(\curl))^3$ such that
    \begin{subequations}\label{eq:structure-preserving-scheme}
    \begin{align}
        (\partial_t {\bm{B}}_h, \bm{C}_h) + (\curl\bm{E}_h, \bm{C}_h) &= 0,  \\
        (\bm{E}_h, \bm{F}_h) + \tau ((\bm{j}_h \times \bm{H}_h) \times \bm{H}_h, \bm{F}_h) &= 0,  \\
        (\bm{j}_h, \bm{K}_h) &= (\bm{B}_h, \curl\bm{K}_h),  \\
        (\bm{H}_h, \bm{D}_h) &= (\bm{B}_h, \bm{D}_h),
    \end{align}
    \end{subequations}
    at all times $t \in \mathbb{R}_+$ and for all $(\bm{C}_h, \bm{F}_h, \bm{K}_h, \bm{D}_h) \in H_0^h(\div) \times H_0^h(\curl)^3$.
\end{problem}

% \pftodo{Where in this problem is the initial value of the timestep enforced? This is a Bubnov--Galerkin scheme (test and trial spaces the same), but I think it should be a Petrov-Galerkin scheme (test and trial spaces distinct). In particular I think that $\bm{C}_h$ should live in $\mathbb{P}_{S-1}(T_n;H_0^h(\div))$ and the initial value should be enforced in the space for $\bm{B}_h$.}

% \pftodo{I'm not sure I like this presentation, where the scheme parachutes in from the sky, unmotivated. Should we motivate its design a little first, before stating it?}

For the time discretization of \eqref{eq:structure-preserving-scheme} we use a Gauss collocation method~\cite[Sec.~II.1.3]{hairer2006geometric}, coinciding with the implicit midpoint method at lowest order. We are thus able to achieve arbitrary order in both time and space.

\subsection{Structure-preserving properties on contractible domains}
Let us first consider contractible $\Omega$, i.e.~domains with trivial topology.

We summarize some properties of our semi-discretization \eqref{eq:structure-preserving-scheme} in the following theorem.
As each of the results therein represents a linear or quadratic structure, they are preserved under the Gauss collocation time discretization~\cite[Sec.~IV.2.1 Theorem 2.1]{hairer2006geometric}.
\begin{theorem}[Structure-preserving properties of the discretisation]\label{thm:sp-properties}
    Any solution to \eqref{eq:structure-preserving-scheme} satisfies:
    \begin{subequations}\label{eq:sp-properties}
    \begin{enumerate}
        \item magnetic advection equation,
        \begin{equation}\label{eq:magnetic-advection-discrete}
            \partial_t\bm B_h + \curl\bm E_h = 0,
        \end{equation}
        \item magnetic Gauss law,
        \begin{equation}\label{eq:gauss}
            \div \bm B_h = \div \bm B_0 = 0.
        \end{equation}
    \end{enumerate}
    Since the magnetic Gauss law \eqref{eq:gauss} holds exactly, and the subcomplex \eqref{eq:subcomplex} is exact on contractible domains, there exists a magnetic potential $\bm A_h \in H_0^h(\curl)$ satisfying $\curl\bm A_h = \bm B_h$.
    With the discrete energy $\mathcal{E}_h \coloneqq \|\bm B_h\|^2$ and helicity $\mathcal{H}_h \coloneqq(\bm A_h, \bm B_h)$, we have the following:
    \begin{enumerate}
        \item[3.] energy dissipation,
        \begin{equation}\label{eq:energy-dissipation}
            \mathcal{E}_h - \mathcal{E}_h|_{t=0}  =  -2 \tau\int_0^t \|\bm{j}_h\times \bm{H}_h\|^2  \le  0,
        \end{equation}
        \item[4.] helicity conservation,
        \begin{equation}\label{eq:helicity-conservation}
            \mathcal{H}_h - \mathcal{H}_h|_{t=0}  =  0.
        \end{equation}
    \end{enumerate}
    With $\mathcal{H}_h$ well-defined, we further have:
    \begin{enumerate}
        \item[5.] the discrete Arnold inequality,
    \begin{equation}\label{eqn:discrete-arnold-inequality}
                C|\mathcal{H}_h| \le \mathcal{E}_h.
        \end{equation}
    \end{enumerate}
    \end{subequations}
\end{theorem}

\begin{proof}
    We see the magnetic advection equation \eqref{eq:magnetic-advection-discrete} by considering $\bm C_h = \partial_t\bm B_h + \curl\bm E_h (\in H_0^h(\div))$ (see e.g.~\cite{huStableFiniteElement2017}).
    The magnetic Gauss law \eqref{eq:gauss} is an immediate consequence, as $\div\circ\curl = 0$.
    For the semi-discrete energy dissipation law \eqref{eq:energy-dissipation},
    \begin{subequations}
    \begin{alignat}{3}
        \mathcal{E}_h - \mathcal{E}_h|_{t=0} &= 2 \int_0^t (\partial_t\bm{B}_h, \bm{B}_h) = - 2 \int_0^t (\curl \bm E_h, \bm B_h)  \\
        &= - 2 \int_0^t (\bm E_h, \curl \bm B_h) = - 2 \int_0^t (\bm E_h, \bm j_h)  \\
        &= - 2\tau \int_0^t \|\bm j_h \times \bm H_h\|^2 \le 0,
    \end{alignat}
    \end{subequations}
    where the second equality holds by the magnetic advection equation \eqref{eq:magnetic-advection-discrete}, and the fourth and last hold by considering $\bm K_h = \bm E_h (\in H_0^h(\curl))$ and $\bm F_h = \bm j_h (\in H_0^h(\curl))$ respectively.
    {The boundary term arising in integration by parts in the second equality vanishes due to the boundary condition $\bm E_h\times \bm n=\bm 0$.}
    For the semi-discrete helicity conservation law \eqref{eq:helicity-conservation},
    \begin{subequations}
    \begin{align}
        \mathcal{H}_h - \mathcal{H}_h|_{t=0} &= 2 \int_0^t (\partial_t\bm{B}_h, \bm{A}_h) = - 2 \int_0^t (\curl \bm E_h, \bm A_h) \\
        &= - 2 \int_0^t (\bm E_h, \bm B_h) = - 2 \int_0^t (\bm E_h, \bm H_h) \\
        &= 2\tau\int_0^t((\bm j_h\times \bm H_h)\times \bm H_h, \bm H_h) = 0
    \end{align}
    \end{subequations}
    where again the second equality holds by the magnetic advection equation \eqref{eq:magnetic-advection-discrete}, and the fourth and sixth hold by considering $\bm D_h = \bm E_h (\in H_0^h(\curl))$ and $\bm F_h = \bm H_h (\in H_0^h(\curl))$ respectively.
    The proof of the discrete Arnold inequality \eqref{eqn:discrete-arnold-inequality} is identical to that of the continuous case \cite{arnold1974asymptotic}.
\end{proof}

\begin{corollary}[Boundedness of the discrete energy]
    The discrete energy remains in a bounded interval:
    \begin{equation}
        C|\mathcal{H}_h| \le \mathcal{E}_h \le \mathcal{E}_h|_{t=0}.
    \end{equation}
\end{corollary}

%\bdainline{This is nice, but I don't really get what it provides...?}
\subsection{Generalized helicity and structure-preserving properties on non-contractible domains}
Aspects of the topology of the domain $\Omega$ are encoded in the de~Rham complex \eqref{eq:deRhamcomplex}.
Over contractible, or topologically trivial, domains (Betti numbers $(\beta_0, \beta_1, \beta_2, \beta_3) = (1, 0, 0, 0)$), for example, the de~Rham complex is exact.
Over domains with a single pair of periodic boundary conditions\footnote{That is, domains that are topologically equivalent to solid tori.} (Betti numbers $(\beta_0, \beta_1, \beta_2, \beta_3) = (1, 1, 0, 0)$), the de~Rham complex is no longer exact at $H_0(\div)$;
namely, there exists over such domains a (non-zero) constant field $\bm B_H \in H_0(\div)$, for which there does not exist a potential $\bm A_H \in H_0(\curl)$ satisfying $\curl\bm A_H = \bm B_H$.\footnote{In fact, since $\div \bm B_H = 0$ and $\curl \bm B_H = \bm 0$, this field $\bm B_H$ is a harmonic 2-form.}
The Parker conjecture is often studied over such domains (see \cite[Fig.~1]{pontinParkerProblemExistence2020}).
However, the definition of the helicity $\mathcal{H}$ \eqref{eq:helicity} relies on the exactness of the de~Rham complex at $H_0(\div)$ (Betti number $\beta_1 = 0$) for the existence of a magnetic potential $\bm A \in H_0(\curl)$;
in the presence of such topological irregularities, it is no longer well-defined.

In this section, we no longer assume $\Omega$ to be topologically trivial.
In particular, we shall allow the de~Rham complex \eqref{eq:deRhamcomplex} not to be exact at $H_0(\div)$ (no conditions on Betti number $\beta_1$).
This necessitates the definition of a generalized helicity (see \eqref{eqn:generalized-helicity} below) that accounts for the harmonic forms introduced by the nontrivial topology.
We shall, however, continue to require exactness at $H_0(\curl)$ (Betti number $\beta_2 = 0$), i.e.~for any $\bm A\in H_{0}(\curl)$ satisfying $\curl\bm A = \bm 0$ there exists $\varphi \in H^{1}_{0}$ such that $\bm A = \grad \varphi$;
this is necessary for the gauge invariance of our generalized helicity (see \Cref{thm:gauge} below) and holds on the domain of current interest:
the box with a single pair of periodic boundary conditions.

Let us first consider the continuous case.
By the Hodge decomposition \cite[Sec.~4.2]{ArnoldFiniteElementExterior2018}, we can decompose a divergence-free magnetic field $\bm B \in H_0(\div)$ into 
\begin{equation}\label{Bhodge}
    \bm B = \curl \bm A + \bm B_{H},
\end{equation}
where $\bm A \in H_0(\curl)$, and $\bm B_{H}$ is a harmonic 2-form ($\div \bm B_H = 0$ and $\curl \bm B_H = \bm 0$).
Given $\bm B$, the decomposition \eqref{Bhodge} determines $\curl \bm A$ and $\bm B_H$ uniquely.\footnote{Note that, while $\curl \bm A$ is determined uniquely, $\bm A$ is not.}
% The field $\bm B_H$ necessarily satisfies the $L^{2}$-orthogonality
% \begin{equation}\label{BRBH}
% \int_{\Omega}\bm B_{R}\cdot \bm B_{H}=0.
% \end{equation}

\begin{theorem}[Invariance of the harmonic component]\label{thm:harmonic-form-constant}
    The harmonic component $\bm B_H$ of $\bm B$ \eqref{Bhodge} remains constant in time in the evolution determined by the magneto-frictional equations \eqref{eqn:magneto-frictional-equations}. 
\end{theorem}

\begin{proof}    
    Substituting the Hodge decomposition for $\bm B$ \eqref{Bhodge} into the magnetic advection equation \eqref{eqn:magnetic-advection},
    \begin{equation}\label{eqn:magnetic-adv}
        \bm 0
        = \partial_t \bm B + \curl \bm E
        = \partial_t [\curl \bm A + \bm B_H] + \curl \bm E
        = \curl [\partial_t \bm A + \bm E] + \partial_t \bm B_H.
    \end{equation}
    By the uniqueness of the Hodge decomposition, both $\curl[\partial_t \bm A + \bm E] = \bm 0$ and $\partial_t \bm B_H = \bm 0$, i.e.~in the latter case the harmonic component $\bm B_H$ is constant.
\end{proof}

With the Hodge decomposition of $\bm B$ \eqref{Bhodge} established, our proposed generalized helicity is defined as follows.

\begin{definition}[Generalized helicity]
    With the Hodge decomposition of $\bm B$ \eqref{Bhodge}, we define a generalized helicity $\tilde{\mathcal{H}}$ by
    \begin{equation}\label{eqn:generalized-helicity}
        \Tilde{\mathcal{H}}
        % \coloneqq \int_{\Omega} \bm{A}\cdot (\bm{B}_R+2\bm{B}_H)
        % = \int_{\Omega} \bm{A}\cdot (\bm{B} + \bm{B}_H).
        = (\bm{A}, \bm{B} + \bm{B}_H) \quad
        (= (\bm{A}, \curl \bm{A} + 2\bm{B}_H)).
    \end{equation}
    % where $\bm A$ is any field satisfying $\curl\bm A=\bm B_{R}$. 
\end{definition}

Dependent on the topology of $\Omega$, we see this definition is gauge-invariant

\begin{theorem}[Gauge invariance]\label{thm:gauge}
    Again assuming the de~Rham complex \eqref{eq:deRhamcomplex} is exact at $H_0(\curl)$ (Betti number $\beta_2 = 0$), the definition of the generalized helicity $\Tilde{\mathcal{H}}$ \eqref{eqn:generalized-helicity} is gauge-invariant.
\end{theorem}

\begin{proof}
    % In contrast to the classical notion of helicity, the generalized helicity $\tilde{\mathcal{H}}_{\bm A}$ may exhibit gauge dependence (i.e.~depend on the specific choice of $\bm A$) depending on the topology of the domain $\Omega$.
    The proof holds similarly to the non-generalized case.
    Consider $\bm A$ through its Hodge decomposition,
    \begin{equation}\label{Ahodge}
        \bm A = \grad \varphi + \curl \bm \psi,
    \end{equation}
    for some $\varphi \in H^1_0$, $\bm \psi \in H_0(\div)$.
    The condition $\bm B = \curl \bm A + \bm B_H = \curl^2 \bm\psi + \bm B_H$ determines $\curl\bm\psi$ uniquely, but imposes no restrictions on $\grad\varphi$ (or equivalently $\varphi$);
    this dictates the choice of gauge on $\bm A$. 
    However, different choices of gauge, i.e.~different choices of $\grad\varphi$, do not affect $\Tilde{\mathcal{H}}$, as $\grad\varphi$ does not contribute to the helicity:
    \begin{equation}
        % (\grad\varphi, \bm{B}_R+2\bm{B}_H) = -(\varphi, \div(\bm{B}_R+2\bm{B}_H)) = 0
        (\grad\varphi, \bm B + \bm B_H) = - (\varphi, \div (\bm B + \bm B_H)) = 0,
    \end{equation}
    {where the boundary term vanishes due to the boundary condition $\varphi=0$ for $\varphi\in H_0^1$.}
\end{proof}

We show now that $\tilde{\mathcal{H}}$ is an invariant under the relaxation process, similar to $\mathcal{H}$.

\begin{theorem}[Invariance of the generalized helicity]\label{thm:conservation-general-helicity}
    The generalized helicity $\tilde{\mathcal{H}}$ is constant in time.
\end{theorem}

\begin{proof}
    Consider first $\partial_t\bm A$.
    From \eqref{eqn:magnetic-adv}, $\curl(\partial_t\bm A + \bm E) = \bm 0$;
    by the Hodge decomposition then, there exists $\varphi \in H^1_0$ such that
    \begin{equation}\label{eqn:A}
        \partial_t\bm A + \bm E = \partial_t \bm{A} - \bm{u}\times \bm{B} = \grad\varphi.
    \end{equation}
    
    Considering the evolution of $\tilde{\mathcal{H}}$,
    \begin{multline}
        \partial_t\tilde{\mathcal{H}}
      = (\partial_t\bm A, \curl \bm A + 2\bm B_H) + (\bm A, \partial_t\curl \bm A)  \\
      = (\partial_t\bm A, \curl \bm A + 2\bm B_H) + (\partial_t\bm A, \curl \bm A)
      = 2(\partial_t\bm A, \bm B)
    \end{multline}
    where in the first equality we use the invariance of $\bm B_H$ (\Cref{thm:harmonic-form-constant}) and in the second we apply integration by parts.
    %With $\bm A = \bm A_0$,
    Using \eqref{eqn:A}, we have
    \begin{equation}
        \partial_t\tilde{\mathcal{H}}
      = 2(\bm u \times \bm B + \grad\varphi, \bm B)
      = - 2(\varphi, \div \bm B)
      = 0,
    \end{equation}
    where we use the orthogonality of the cross product, and that the boundary term vanishes due to the boundary condition $\varphi=0$ for $\varphi\in H_0^1$.
\end{proof}

We conclude by showing that a generalized version of the Arnold inequality holds.
%For domains with nontrivial harmonic 1-forms (i.e.~where the generalized helicity $\tilde{\mathcal{H}}_{\bm A}$ is gauge-dependent) this holds under a certain choice of $\bm A$;
%for domains with no nontrivial harmonic 1-forms (e.g.~our domain of interest, the periodic box) the helicity is gauge-invariant and the discrete Arnold inequality \eqref{eq:generalized-discrete-arnold} holds in general.

\begin{theorem}[Generalized Arnold inequality]\label{thm:generalized_arnold}
    % Let $\bm{A} \in H_0(\curl)$ be the magnetic potential (satisfying $\bm B = \curl \bm A + \bm B_H$ \eqref{Bhodge}) with no harmonic components, i.e.~with Hodge decomposition $\bm A = \grad \varphi + \curl \bm \psi$ for some $\varphi$ and $\bm \psi$.
    % There then exists some constant $C > 0$ (dependent only on the domain $\Omega$) such that
    % \begin{equation}\label{eq:generalized-discrete-arnold}
    %         |\Tilde{\mathcal{H}}|  \le  C^{-1}\mathcal{E}.
    % \end{equation}
    There exists some constant $C > 0$ (dependent only on the domain $\Omega$) such that
    \begin{equation}\label{eq:generalized-discrete-arnold}
            |\Tilde{\mathcal{H}}|  \le  C^{-1}\mathcal{E}.
    \end{equation}
\end{theorem}

\begin{proof}
    As in \eqref{Ahodge}, let $\bm A$ have Hodge decomposition $\bm A = \grad \varphi + \curl \bm \psi$.
    As $\Tilde{\mathcal{H}}$ is independent of $\varphi$, let us assume without loss of generality $\varphi = 0$.
    Applying the Cauchy--Schwarz, generalized Poincar\'e and Young's inequality respectively, there exists some constant $C > 0$ such that
    \begin{subequations}
    \begin{multline}
            |\tilde{\mathcal{H}}|
        \le \|\bm A\|\|\curl \bm A + 2\bm B_H\|
        \le (2C)^{-1}\|\curl \bm A\|\|\curl \bm A + 2\bm B_H\|  \\
        \le (4C)^{-1}(\|\curl \bm A\|^2 + \|\curl \bm A + 2\bm B_H\|^2).
    \end{multline}
    Since $\bm B_H$ is a harmonic 2-form, it is $L^2$-orthogonal to $\curl \bm A$, implying
    \begin{equation}
            |\tilde{\mathcal{H}}|
        \le (4C)^{-1}(\|\curl \bm A\|^2 + \|\curl \bm A + 2\bm B_H\|^2)
        \le C^{-1}\|\bm B\|^2.
    \end{equation}
    \end{subequations}
\end{proof}

Considering our proposed semi-discretisation \eqref{eq:structure-preserving-scheme}, similar arguments to those in the proof of \Cref{thm:sp-properties} hold under these nontrivial topologies, for the analogous results presented above.
In summary, the discrete magnetic field $\bm B_h$ has Hodge decomposition $\bm B_h = \curl \bm A_h + \bm B_{Hh}$ for some $\bm A_h \in H_0^h(\curl)$.
The discrete harmonic component $\bm B_{Hh}$ remains constant in time, the discrete generalized helicity $\tilde{\mathcal{H}}_h$ remains constant, %under an approriate choice of $\bm A_h = \bm A_{0h}$,
and the generalized discrete Arnold inequality \eqref{eq:generalized-discrete-arnold} holds;
these results transfer to the fully discrete case when using a Gauss collocation method for the time discretization.
We have therefore the same quantitative and qualitative guarantees for numerical solutions derived from our discrete scheme on domains with nontrivial topologies.

\begin{remark}[Domains with nontrivial harmonic 1-forms]\label{rmk:harmonic-1forms}
    We may still define the generalized helicity \eqref{eqn:generalized-helicity} where the de~Rham complex is not exact at $H_0(\curl)$ (Betti number $\beta_2 \ne 0$), however the definition is gauge-dependent.
    To see this, consider the Hodge decomposition
    \begin{equation}
        \bm A = \grad \varphi + \curl \bm \psi + \bm A_H,
    \end{equation}
    where $\bm A_H$ is a harmonic 1-form;
    while $\bm A_H$ is not determined by the condition $\bm B = \curl \bm A + \bm B_H$, different choices of $\bm A_H$ give rise to different values of $\tilde{\mathcal{H}}$.
    Due to the gauge-dependence, the statement of \Cref{thm:conservation-general-helicity} may change to:
    there exists a certain choice of magnetic potential $\bm A = \bm A_1 \in H_0(\curl)$ (satisfying $\bm B = \curl \bm A + \bm B_H$) such that the generalized magnetic helicity $\tilde{\mathcal{H}}$ is constant.
    The statement of \Cref{thm:generalized_arnold} may change to:
    there exists a certain choice of magnetic potential $\bm A = \bm A_2$ such that the generalized Arnold inequality \eqref{eq:generalized-discrete-arnold} holds.
    However, these potentials $\bm A_1$ and $\bm A_2$ may differ.
    Similar arguments hold on the discrete level.
\end{remark}

\begin{remark}
    There have been many efforts in the literature to generalize the concept of helicity to topologically nontrivial domains.
    See, for example, the notions of relative helicity \cite{ArnoldTopologicalMethodsHydrodynamics2021} and the universal magnetic helicity integral proposed in \cite{hornigUniversalMagneticHelicity2006}.
    
    For domains $\Omega$ embedded in $\mathbb{R}^3$, many approaches use a potential $\hat{\bm A} \in H(\curl)$ such that $\bm B = \curl \hat{\bm A}$, without enforcing the boundary condition $\hat{\bm A} \times \bm n = \bm 0$.
    Gauge invariance is then recovered by introducing a certain boundary integral in $\hat{\bm A}$ on $\partial\Omega$.
    Over domains topologically equivalent to tori, the Bevir--Gray helicity \cite{bevir1980relaxation} is defined as
    \begin{equation}\label{eq:bevir-gray}
        \mathcal{H}_{\rm BG} \coloneqq \int_{\Omega} \hat{\bm A}\cdot \bm B - \oint_{\gamma_A} \hat{\bm A}\cdot \bm t_A\oint_{\gamma_B}\hat{\bm A}\cdot \bm t_B,
    \end{equation}
    where $\gamma_A$ and $\gamma_B$ are closed paths on $\partial\Omega$ along the minor (A-cycle) and major (B-cycle) circumferences, and $\bm t_A$, $\bm t_B$ are their respective unit tangents. McTaggart \& Valli \cite{mactaggart2019magnetic} extended this definition to more general topologies.
    
    The Bevir--Gray construction addresses the same topological difficulty as our generalized helicity $\tilde{\mathcal{H}}$ \eqref{eqn:generalized-helicity}, but by a different mechanism: it uses cutting surfaces to define the necessary integrals, whereas our approach uses harmonic forms. The two settings are therefore not straightforward to compare. In particular, in our definition the vector potential $\bm A$ is chosen as the preimage of the curl-range component of $\bm B$, while in the Bevir--Gray framework it appears that $\hat{\bm A}$ is given a priori and $\bm B$ is then defined by $\curl \hat{\bm A}$. Clarifying the precise correspondence between these viewpoints would require further investigation.
\end{remark}

\section{Numerical experiments}\label{sec:Numerical-experiment}
To demonstrate the benefits provided by the structures preserved by our proposed scheme, we perform various simulations of~\eqref{eqn:magneto-frictional-equations} under two different ICs, i.e.~Hopf fibration \cite{smietIdealRelaxationHopf2017} and IsoHelix \cite{pontin2009lagrangian,candelaresiMimeticMethodsLagrangian2014a}, which are commonly tested for Lagrangian methods in magnetic relaxation. Each case are tested on the domain $\Omega = (-4, 4)\times (-4, 4)\times (-10, 10)$ with fixed coupling parameter $\tau = 100$. For the spatial discretization, we employ in each case the lowest-order N\'ed\'elec finite elements of the first kind and the lowest-order Raviart--Thomas finite element in the same de~Rham complex. We use the lowest order Gauss method, i.e. the implicit midpoint rule, for the time discretization. Unless specified otherwise, we use a coarse mesh with $4\times 4\times 10$ hexahedral cells in the $x \times y \times z$ directions, and a fixed uniform step size $\Delta t = 10$ with final time $T = 10^4$. For a domain with trivial topology, Dirichlet boundary conditions are imposed on each face. For a domain with nontrivial topology, Dirichlet boundary conditions are imposed on the side faces and periodic conditions are imposed on the top and bottom in the $z$-direction. In our implementation, we use Newton's method to solve the nonlinear systems, and a sparse LU factorization to solve the arising linear systems.

\subsection{Projection of ICs and evaluation of helicity}\label{sec:helicity_evaluation}
The exact divergence-free condition on $\bm B_h$ is crucial for our structure-preserving properties (\Cref{thm:sp-properties}).
For each simulation, the projection of the continuous initial data $\bm B_0$ to the divergence-free subspace of $H_0^h(\div)$ necessitates the solution of a saddle-point problem:
{find $(\bm B_{0,h}, p_h)\in H_0^h(\div)\times L_0^{2,h}$ such that
\begin{subequations}
\begin{align}
    (\bm B_{0,h}, \bm C_h) - (p_h, \nabla\cdot \bm C_h) &= (\bm B_0, \bm C_h), \\
    (q_h, \nabla\cdot \bm B_{0,h}) &= 0,
\end{align}
\end{subequations}
for all $(\bm C_h, q_h)\in H_0^h(\div)\times L_0^{2,h}$, where $\bm B_{0, h} \in H_0^h(\div)$ represent our discrete, divergence-free ICs.}

The divergence-free condition is required in particular for the discrete helicity $\mathcal H_h$ to be well-defined.
At each timestep, we evaluate $\mathcal H_h$ for a given magnetic field configuration $\bm{B}_h$ by finding $\bm{A}_{h}\in H_0^h(\curl)$ such that 
\begin{equation}\label{eqn:evaluate-magnetic-potential}
    (\curl\bm{A}_{h}, \curl\bm{C}_h)  =  (\bm{B}_{h}, \curl\bm{C}_h)  \quad
    \forall \bm{C}_h\in H_0^h(\curl).
\end{equation}
This system is symmetric and consistent, but singular. Therefore, we use MINRES with the minimum norm refinement strategy of Liu \textit{et al.}~\cite{liu2025obtaining} to find the pseudoinverse solution. This is used to define the discrete helicity $\mathcal H_h$.

\subsection{Hopf fibration (non-zero helicity)}

The former initial configuration we consider is the Hopf fibration \cite{smietIdealRelaxationHopf2017},
\begin{equation}
    \bm{B}_0
    =
    \frac{4 \sqrt{s}}{\pi\left(1+r^2\right)^3 \sqrt{\omega_1^2+\omega_2^2}}\left(\begin{array}{c}
        2\left(\omega_2 y-\omega_1 x z\right) \\
        -2\left(\omega_2 x+\omega_1 y z\right) \\
        \omega_1\left(-1+x^2+y^2-z^2\right)
    \end{array}\right),
    \label{eqn:hopf-fibre}
\end{equation}
where $\omega_1, \omega_2 \in \mathbb R$ are winding numbers, and $s \ge 0$ is a scaling parameter.
We choose $\omega_1 = 3$, $\omega_2 = 2$, $s = 1$, such that the field lines form three windings in the poloidal direction for every two in the toroidal direction, thus exhibiting a non-zero helicity. 

\subsubsection{Long-time preservation test}
\Cref{fig:structure-preserving-scheme} shows the preserved helicity and decreasing energy for domains with the trivial and nontrivial topology, bounded below in both cases by a constant multiple of the helicity (\ref{eqn:arnold-inequality}, \ref{eq:generalized-discrete-arnold}). 
\Cref{fig:helicity-error-comparison} shows the evolution of the helicity and the divergence of the magnetic field, demonstrating both are conserved for long times. 
\begin{figure}[htpb]
    \centering
    \includegraphics[width=1.0\linewidth]{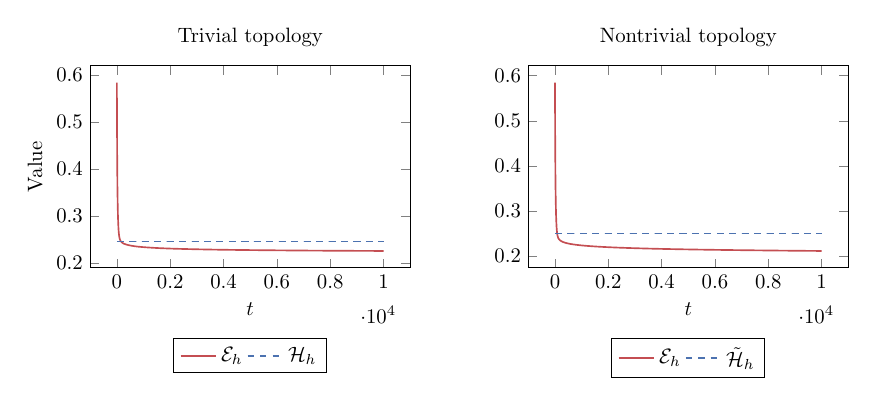}
    \caption{Discrete energy dissipation \eqref{eq:energy-dissipation} and helicity conservation \eqref{eq:helicity-conservation} for a comparable trivial and nontrivial topology, under our structure-preserving scheme \eqref{eq:structure-preserving-scheme}}
    \label{fig:structure-preserving-scheme}
\end{figure}

\begin{figure}[htpb]
    \centering
    \includegraphics[width=1.0\linewidth]{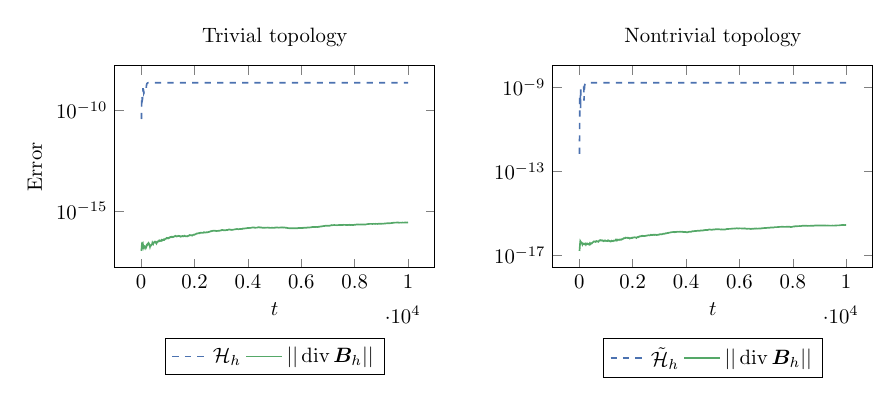}
    \caption{Errors in helicity ($|\mathcal{H}_h - \mathcal{H}_h(0)|$ and $|\tilde{\mathcal{H}}_h(t) - \tilde{\mathcal{H}}_h(0)|$ respectively) and $\|\div \bm B_h\|$ for a comparable domain with trivial and nontrivial topology, under our structure-preserving scheme \eqref{eq:structure-preserving-scheme}}
    \label{fig:helicity-error-comparison}
\end{figure}

\subsubsection{Significance of the auxiliary magnetic field}
% The auxiliary variable $\bm{H} = \mathbb{Q}_c \bm{B}$ is essential for helicity preservation.
To what extent is the auxiliary variable $\bm H_h \approx \bm B_h$ introduced by the framework in \cite{andrews2024enforcing} necessary?
%\bda{Weird to have a question?}
The following semi-discrete scheme replaces $\bm H_h$ with $\bm B_h$.

\begin{problem}[$H(\div)$-conforming semi-discrete scheme]
    Find $\bm B_h \in C^1(\mathbb{R}_+; H_0^h(\div))$ satisfying the ICs $\bm B|_{t=0} = \bm B_0$ and $(\bm E_h, \bm j_h) \in C^0(\mathbb{R}_+; H_0^h(\curl)^2)$ such that
    \begin{subequations}\label{eqn:without-H}
    \begin{align}
        (\partial_t {\bm{B}}_h, \bm{C}_h) + (\curl\bm{E}_h, \bm{C}_h) &= 0,  \\
        (\bm{E}_h, \bm{F}_h) + \tau ((\bm{j}_h \times \bm{B}_h) \times \bm{B}_h, \bm{F}_h) &= 0,  \\
        (\bm{j}_h, \bm{K}_h) &= (\bm{B}_h, \curl\bm{K}_h),
    \end{align}
    \end{subequations}
    at all times $t \in \mathbb{R}_+$ and for all $(\bm{C}_h, \bm{F}_h, \bm{K}_h) \in H_0^h(\div) \times H_0^h(\curl)^2$.
\end{problem}

% Again, for the time discretization of \eqref{eqn:without-H} we use a Gauss collocation method.
We reapply the techniques of \Cref{sec:helicity_evaluation} to project the ICs $\bm B_0$ into the divergence-free subspace of $H_0^h(\div)$ and to evaluate the helicity.
\Cref{fig:comparison_our_Hdiv} compares the evolution of $\cal E$ and $\cal H$ in the schemes \eqref{eq:structure-preserving-scheme} and \eqref{eqn:without-H}. \Cref{fig:error-our-hdiv} shows the error plot for the helicity and divergence-free condition.  
Without $\bm H_h$, the helicity artificially dissipates, and can no longer prevent the energy from decaying to zero;
the numerical solution from \eqref{eqn:without-H} converges to a non-physical trivial steady state. 
\begin{figure}[htpb]
    \centering    \includegraphics[width=1.0\linewidth]{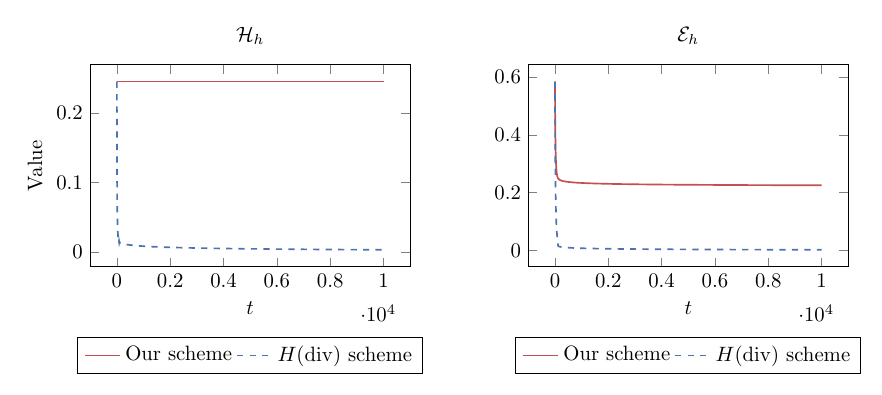}
    \caption{Helicity, $\mathcal{H}_h$, and energy, $\mathcal{E}_h$, for our proposed scheme \eqref{eq:structure-preserving-scheme} and the $H(\div)$-conforming scheme \eqref{eqn:without-H}}
    \label{fig:comparison_our_Hdiv}
\end{figure}
\begin{figure}[htpb]
    \centering
    \includegraphics[width=1.0\linewidth]{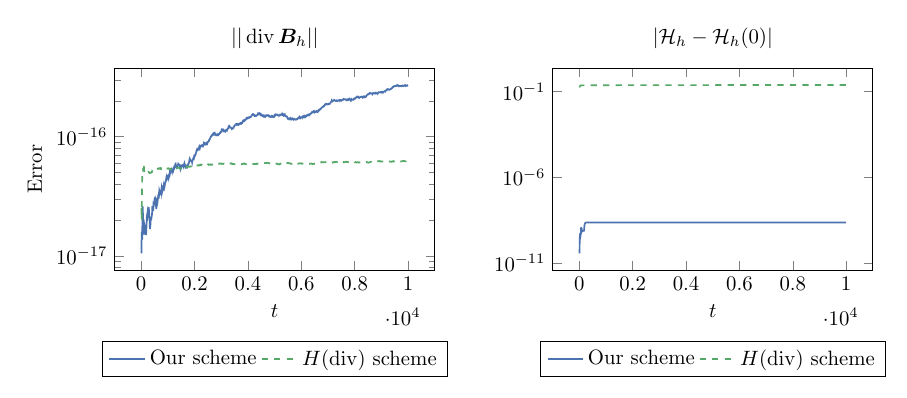}
    \caption{Errors $\|\div \bm B_h\|$ and $|\mathcal{H}_h - \mathcal{H}_h(0)|$ for our proposed scheme \eqref{eq:structure-preserving-scheme} and the $H(\div)$-conforming scheme \eqref{eqn:without-H}}
    \label{fig:error-our-hdiv}
\end{figure}

\subsubsection{Magnetic relaxation with other schemes}
Inspired by classical finite element discretizations for the MHD system \cite{GuzmanConformingdivergencefreeStokes2014, schotzau2004mixed}, we compare our proposal against $H(\curl)$- and $H^1$-conforming discretizations.
We temporarily refine our timestep $\Delta t = 1$, using a shorter final time $T = 1000$, again discretizing in time using the implicit midpoint method.

\begin{problem}[$H(\curl)$-conforming semi-discrete scheme]\label{pb:Hcurl-method}
    Find $\bm B_h \in C^1(\mathbb{R}_+; H^h(\curl))$ satisfying the ICs $\bm B|_{t=0} = \bm B_0$ (assuming $\bm B_0 \in H^h(\curl)$) and $\bm u_h \in C^0(\mathbb{R}_+; H_0^h(\div))$ such that
    \begin{subequations}\label{eqn:curl-scheme}
    \begin{align}
        (\partial_t {\bm{B}}_h, \bm{C}_h)  &=  (\bm u_h \times \bm B_h, \curl\bm{C}_h),  \\
        (\bm u_h, \bm v_h)  &=  \tau ((\curl \bm B_h \times \bm B_h), \bm v_h),
    \end{align}
    \end{subequations}
    at all times $t \in \mathbb{R}_+$ and for all $(\bm{C}_h, \bm{v}_h) \in H^h(\curl) \times H_0^h(\div)$.
\end{problem}

\begin{remark}        
    For \Cref{pb:Hcurl-method}, by integration by parts
    \begin{subequations}
    \begin{equation}
        (\bm u\times \bm B, \curl \bm C) = (\curl(\bm u\times \bm B), \bm C) + \int_{\partial \Omega}(\bm u \times \bm B)\cdot (\bm n\times \bm C).
    \end{equation}
    The boundary term can be rewritten as 
    \begin{equation}
        \int_{\partial \Omega} (\bm u \times \bm B) \cdot (\bm n\times \bm C)
      = \int_{\partial\Omega}[(\bm u\times \bm B)\times \bm n] \cdot \bm C
      = -\int_{\partial\Omega}(\bm B\cdot \bm n)(\bm u\cdot \bm C),
    \end{equation}
    \end{subequations}
    where in the last term we use the boundary condition $\bm u\cdot \bm n = 0$.
    We see then the boundary condition $\bm B\cdot \bm n=0$ is imposed weakly.
\end{remark}

\begin{problem}[$H^1$-conforming scheme]
    Find $\bm B_h \in C^1(\mathbb{R}_+; H^{1,h}\cap H_0(\div))$ satisfying the ICs $\bm B|_{t=0} = \bm B_0$ (assuming $\bm B_0 \in H^{1,h}\cap H_0(\div)$) and $\bm u_h \in C^0(\mathbb{R}_+; H^{1,h}\cap H_0(\div))$ such that \eqref{eqn:curl-scheme} holds at all times $t \in \mathbb{R}_+$ and for all $(\bm{C}_h, \bm{v}_h) \in (H^{1,h}\cap H_0(\div))^2$.
\end{problem}

\begin{figure}[htpb]
     \centering
     \includegraphics[width=1.0\linewidth]{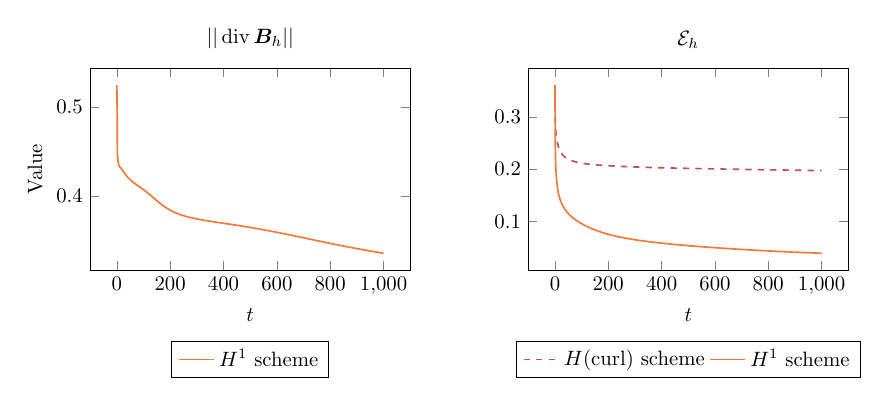}
     \caption{Error $||\div \bm B_h||$ and evolution of $\mathcal{E}_h$ for the $H(\curl)$-- and $H^1$--conforming schemes}
     \label{fig:comparison-hcurl-cg}
\end{figure}

%\bdainline{How did you implement $H^{1,h} \cap H_0(\div)$ in Firedrake? My experience is that Firedrake isn't really equipped for such BCs.}

%\bdainline{Why does the energy in the $H(\curl)$ scheme appear not to be decaying to 0? Does it over a long enough time?}

\Cref{fig:comparison-hcurl-cg} illustrates the divergence of the magnetic field $\|\div \bm B_h\|$ and the dissipation of energy $\mathcal{E}_h$ for the $H(\curl)$- and $H^1$-conforming schemes.
The $H^1$-conforming method fails to preserve the divergence-free condition, while $\div \bm B_h$ is ill-defined for the $H(\curl)$-conforming method as $\bm B_h$ is not $\div$-conforming.
In either case, since $\bm B_h$ is not generally divergence-free, the helicity is not well-defined.
Neither of these methods, therefore, are appropriate for investigating the Parker conjecture.

\subsection{IsoHelix (zero helicity)}
Our latter initial configuration is the IsoHelix~\cite{pontin2009lagrangian,candelaresiMimeticMethodsLagrangian2014a},
\begin{equation}\label{eqn:isoHelix}
    \bm B_0 = \left(\begin{array}{c}
         \alpha(r, z) y \\
         - \alpha(r, z) x \\
         1
    \end{array}\right),
\end{equation}
where $\alpha(r, z) \coloneqq \frac{\pi}{2}z\exp(- \frac{1}{2}r^2 - \frac{1}{4}z^2)$ and $r^2 = x^2 + y^2$.
Such a magnetic field can simply be obtained by twisting a homogeneous field. This field has zero helicity, and will relax to a homogeneous steady state of the form $\bm B = (0, 0, 1)^T$. Since this steady state has an exact closed-form, it allows us to compare the quality of the relaxation in a straightforward way. 

We use our structure-preserving discretization \eqref{eq:structure-preserving-scheme} with periodic boundary conditions in the $z$ direction only. Therefore, the domain has nontrivial topology. We thus monitor the generalized helicity \eqref{eqn:generalized-helicity} and the modified energy $\tilde{\mathcal{E}}_h = \|\bm B_h - \bm B_H\|^2$ as suggested in \cite{candelaresiMimeticMethodsLagrangian2014a}, where the harmonic form $\bm B_H = (0, 0, 1)^T$ is the homogeneous background field.
\Cref{fig:iso_plot} shows that the discrete generalized helicity $\tilde{\mathcal{H}}_h$ remains at $0$.
As a result, the modified energy is monotonically decreasing to $0$ while the harmonic form remains constant according to \Cref{thm:harmonic-form-constant}. \Cref{fig:simulation-isoHelix} demonstrates the evolution of the magnetic field lines as they approach the equilibrium $\bm B = (0, 0, 1)^T$.
\begin{figure}
    \centering
\includegraphics[width=0.5\linewidth]{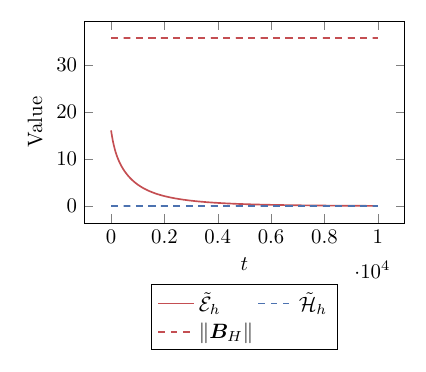}
    \caption{Generalized helicity $\tilde{\mathcal{H}}_h$, modified energy $\tilde{\mathcal{E}}_h$, and harmonic form $\|\bm B_H\|$, under our structure-preserving scheme \eqref{eq:structure-preserving-scheme}}
    \label{fig:iso_plot}
\end{figure}
%\bdainline{Why is $\|\bm B_H\|$ on this plot?}
\begin{figure}[htpb]
    \centering
    \begin{subfigure}{0.25\textwidth}
        \includegraphics[width=\linewidth]{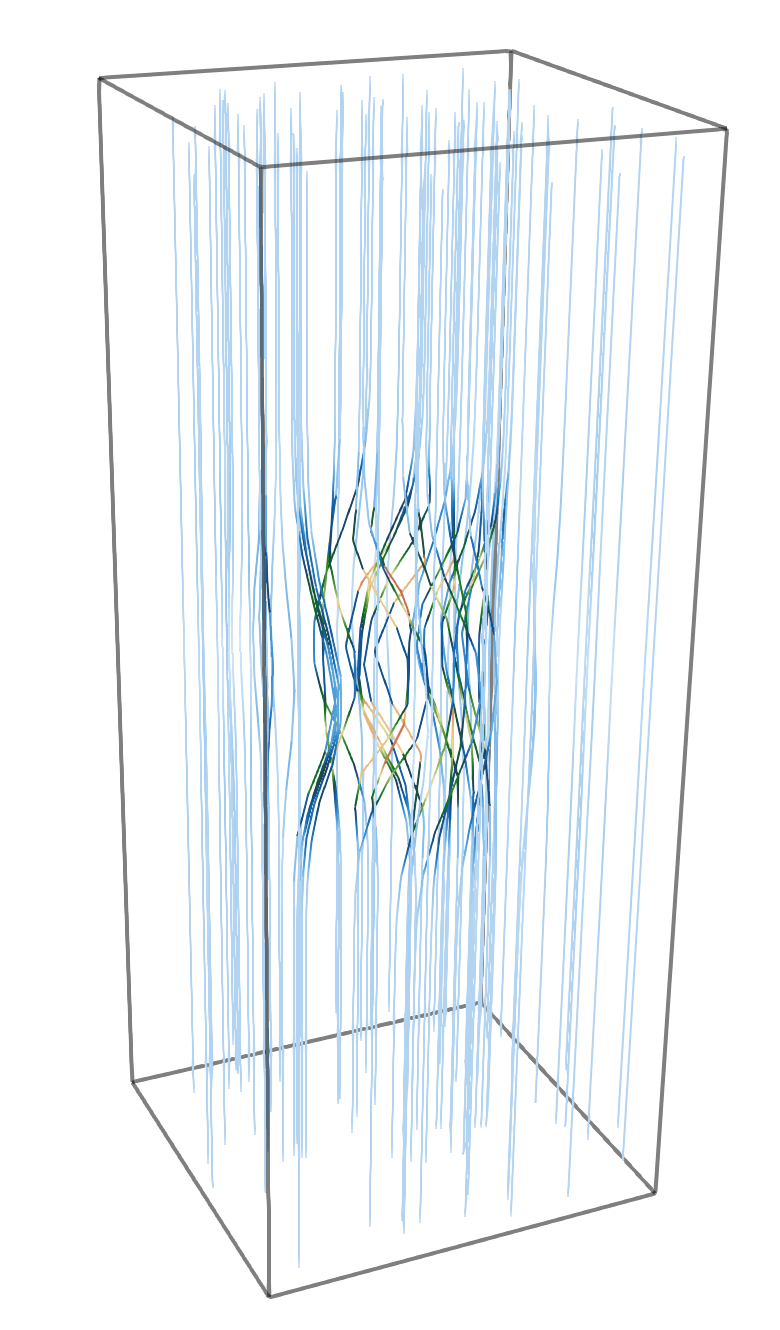}
        \caption{$t=0$}
    \end{subfigure}
    \hfill
    \begin{subfigure}{0.25\textwidth}
        \includegraphics[width=\linewidth]{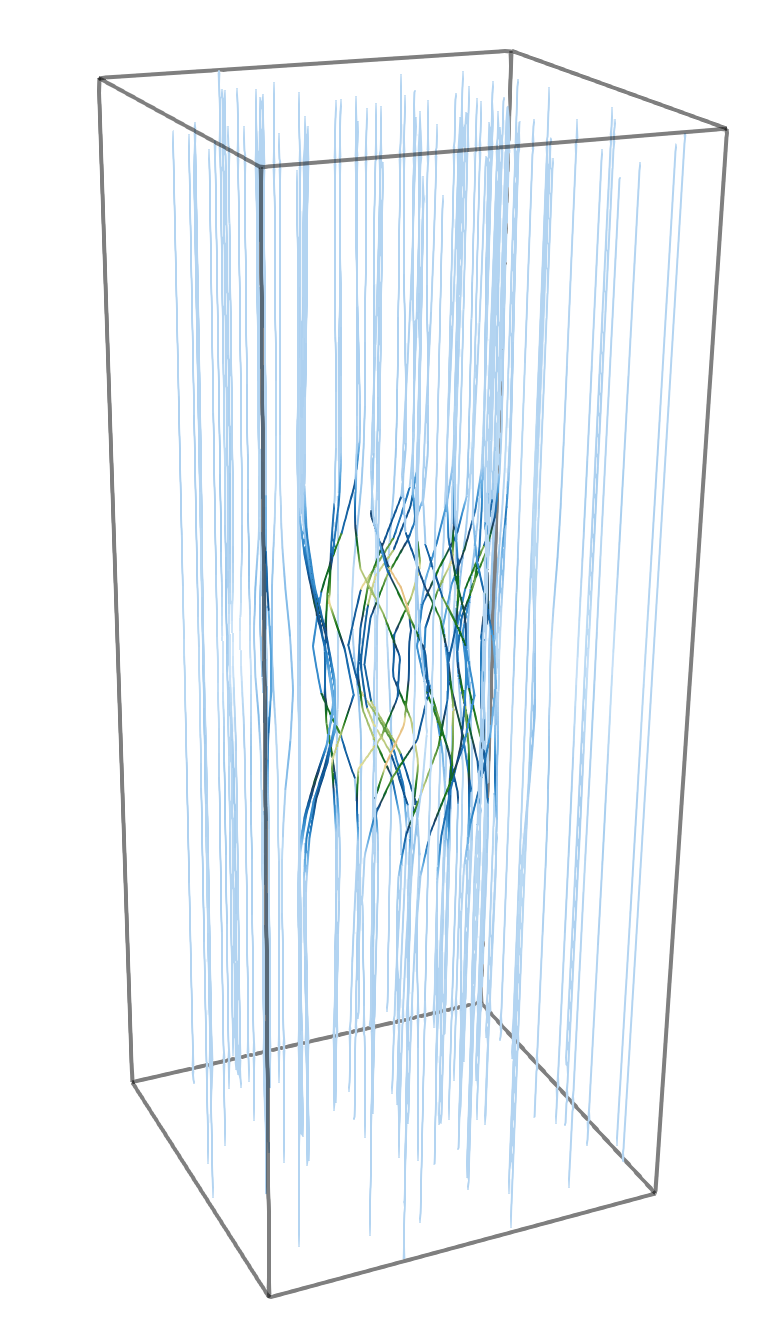}
        \caption{$t=10$}
    \end{subfigure}
    \hfill
    \begin{subfigure}{0.25\textwidth}
        \includegraphics[width=\linewidth]{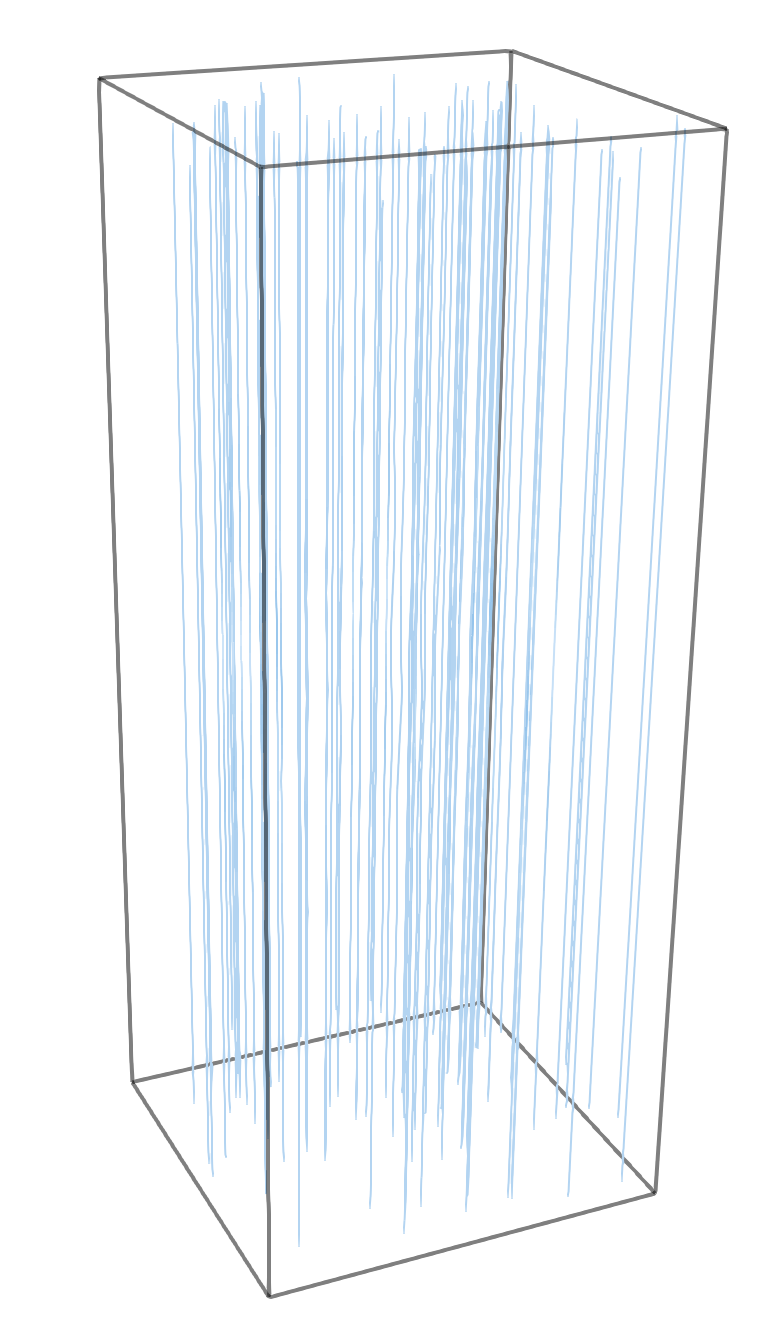}
        \caption{$t=10000$}
    \end{subfigure}
    \hfill
    \begin{subfigure}{0.14\textwidth}
        \includegraphics[width=\linewidth]{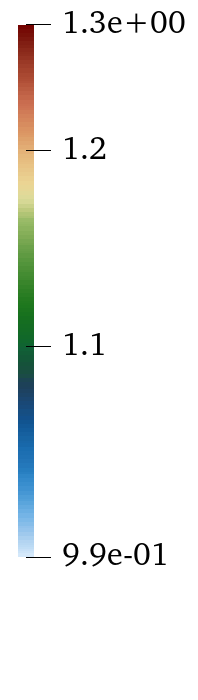}
    \end{subfigure}
    \caption{Magnetic field lines for the IsoHelix simulation on a domain with nontrivial topology, colored by magnetic field strength $\|\bm B_h\|$.}
    \label{fig:simulation-isoHelix}
\end{figure}

\subsection{Larger-scale simulation}
Returning to the longer timestep $\Delta t = 10$, we run our structure-preserving scheme \eqref{eq:structure-preserving-scheme} on a larger-scale problem for the Hopf fibration \eqref{eqn:hopf-fibre}, with a more refined $32\times 32\times 80$ mesh (with 1,026,480 degrees of freedom) on the UK supercomputer ARCHER2 \cite{archer2}.
Figures~\ref{fig:simulation-trivial-domain} and~\ref{fig:simulation-nontrivial-domain} plot cross-sections of the magnetic field lines for the same setup, in domains of trivial and nontrivial topology, respectively.
The numerical results in either case converge to a nontrivial steady state. 
\begin{figure}[htpb]
    \centering
    \begin{subfigure}{0.25\textwidth}
        \includegraphics[width=\linewidth]{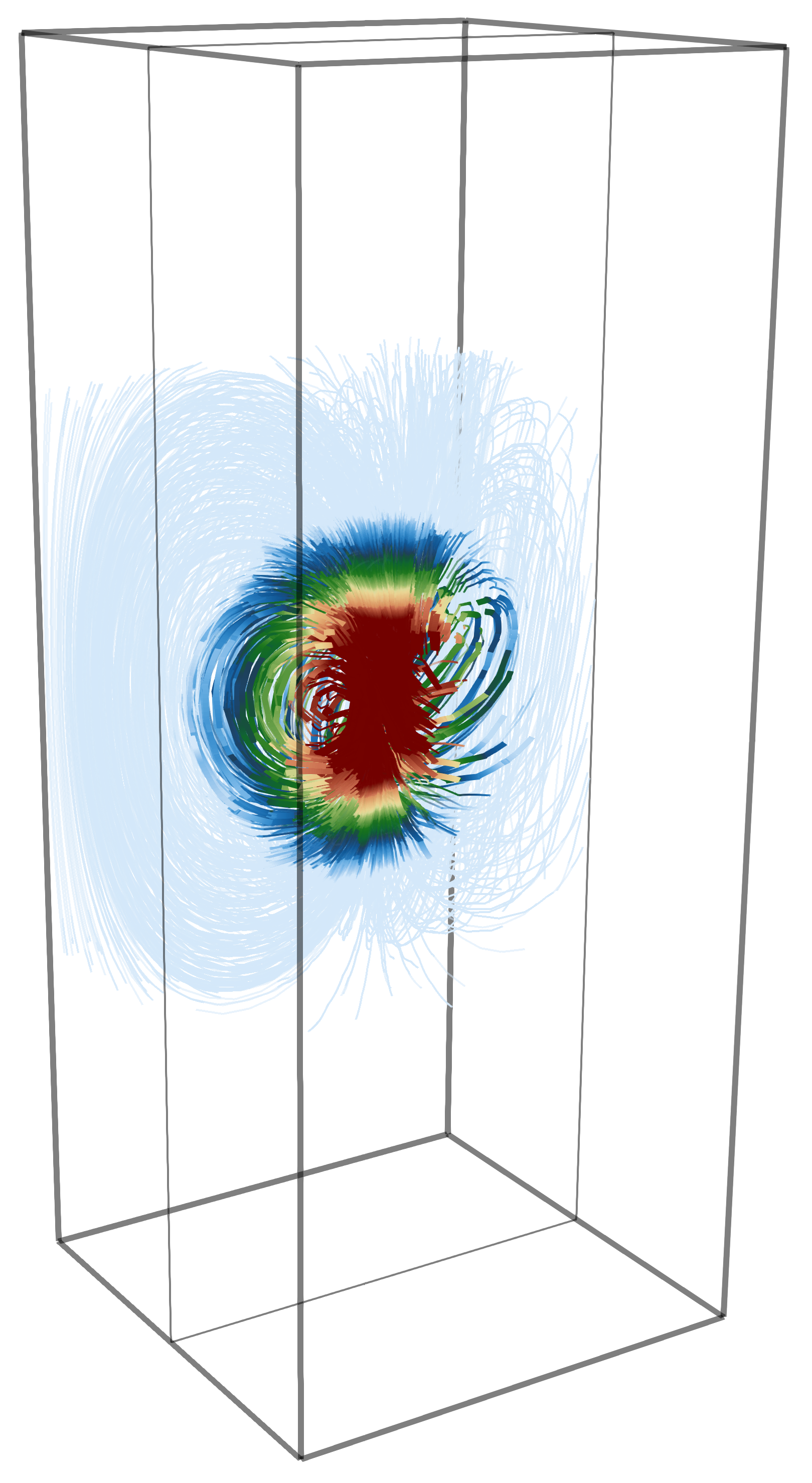}
        \caption{$t=0$}
    \end{subfigure}
    \hfill
    \begin{subfigure}{0.25\textwidth}
        \includegraphics[width=\linewidth]{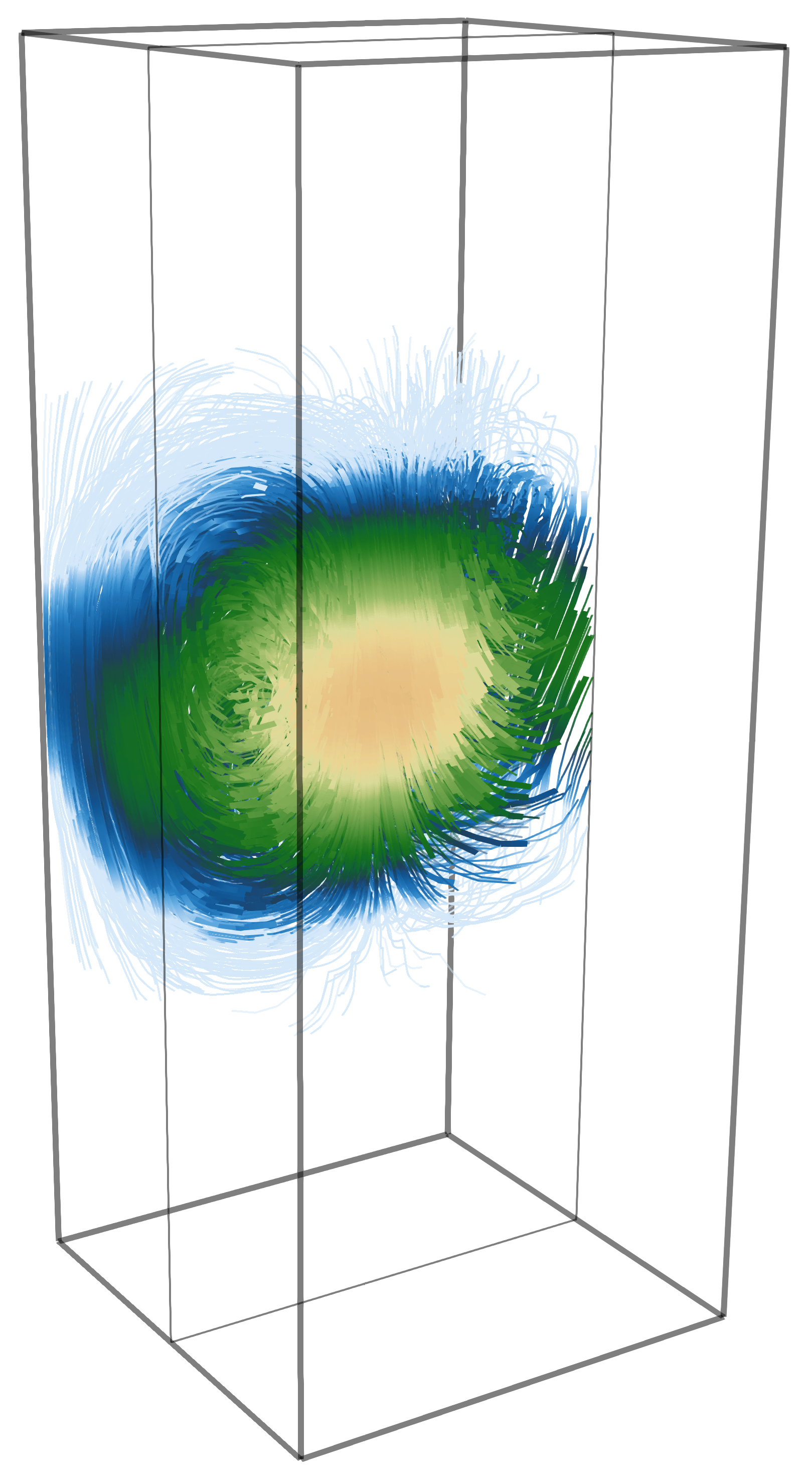}
        \caption{$t=10$}
    \end{subfigure}
    \hfill
    \begin{subfigure}{0.25\textwidth}
        \includegraphics[width=\linewidth]{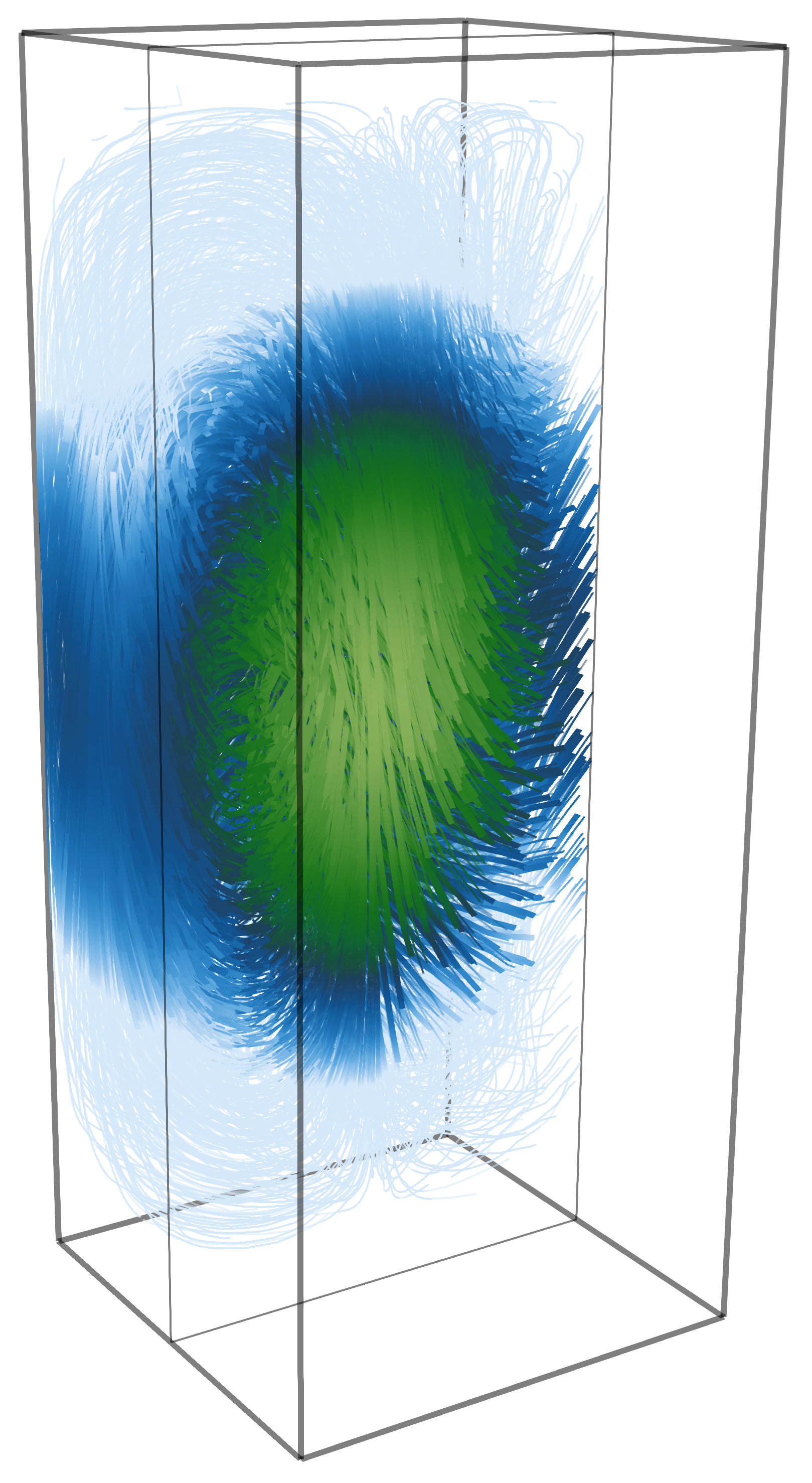}
        \caption{$t=10000$}
    \end{subfigure}
    \hfill
    \begin{subfigure}{0.14\textwidth}
        \includegraphics[width=\linewidth]{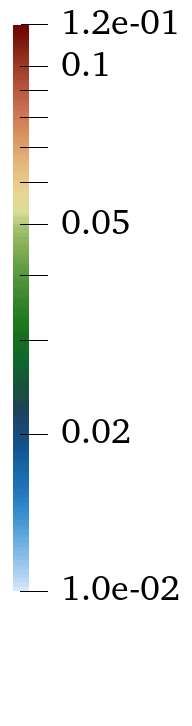}
    \end{subfigure}
    \caption{Magnetic field lines under magnetic relaxation of the Hopf fibration, on a domain with trivial topology, colored by magnetic field strength $\|\bm B_h\|$.}
    \label{fig:simulation-trivial-domain}
\end{figure}

\begin{figure}[htpb]
    \centering
    \begin{subfigure}{0.25\textwidth}
        \includegraphics[width=\linewidth]{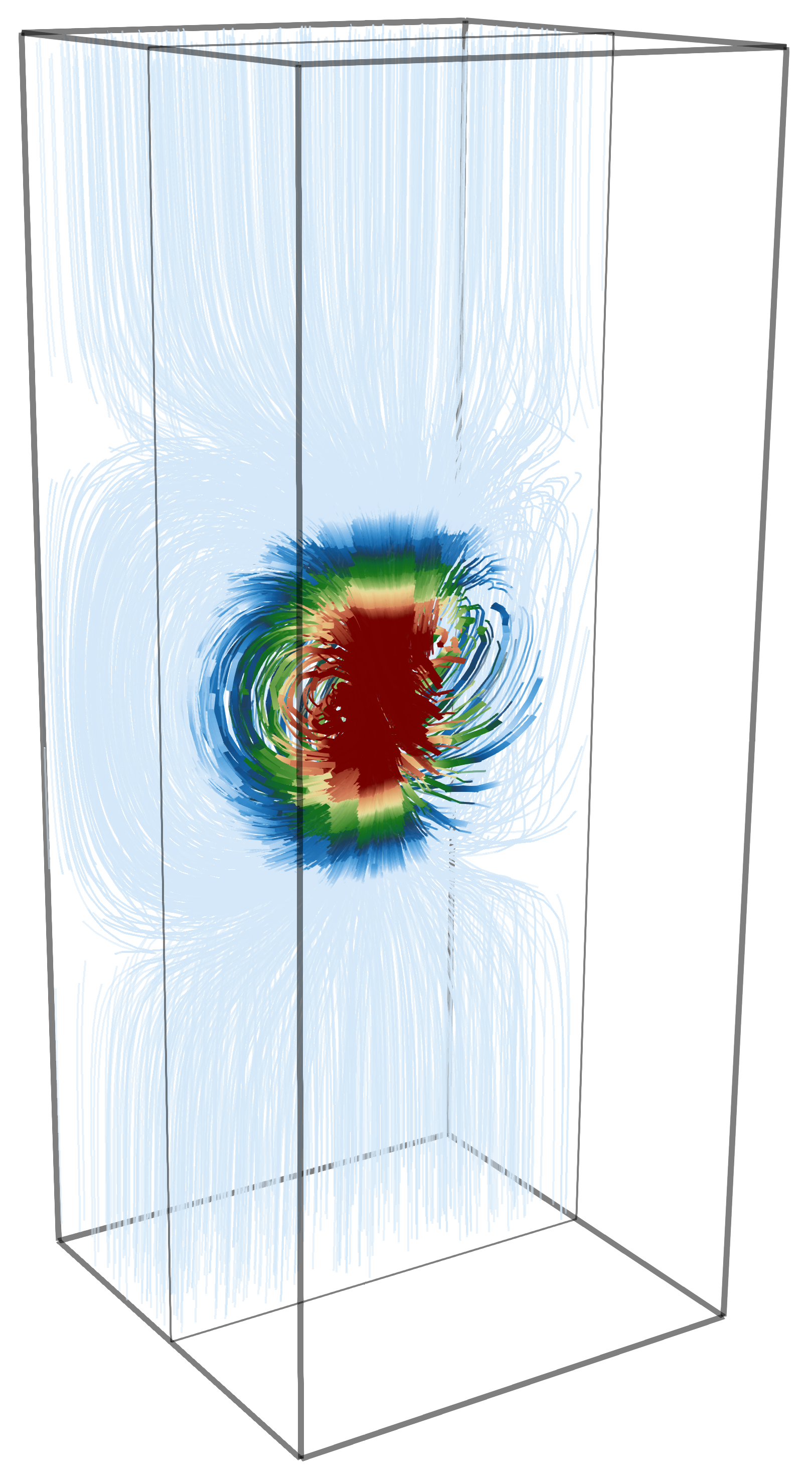}
        \caption{$t=0$}
    \end{subfigure}
    \hfill
    \begin{subfigure}{0.25\textwidth}
        \includegraphics[width=\linewidth]{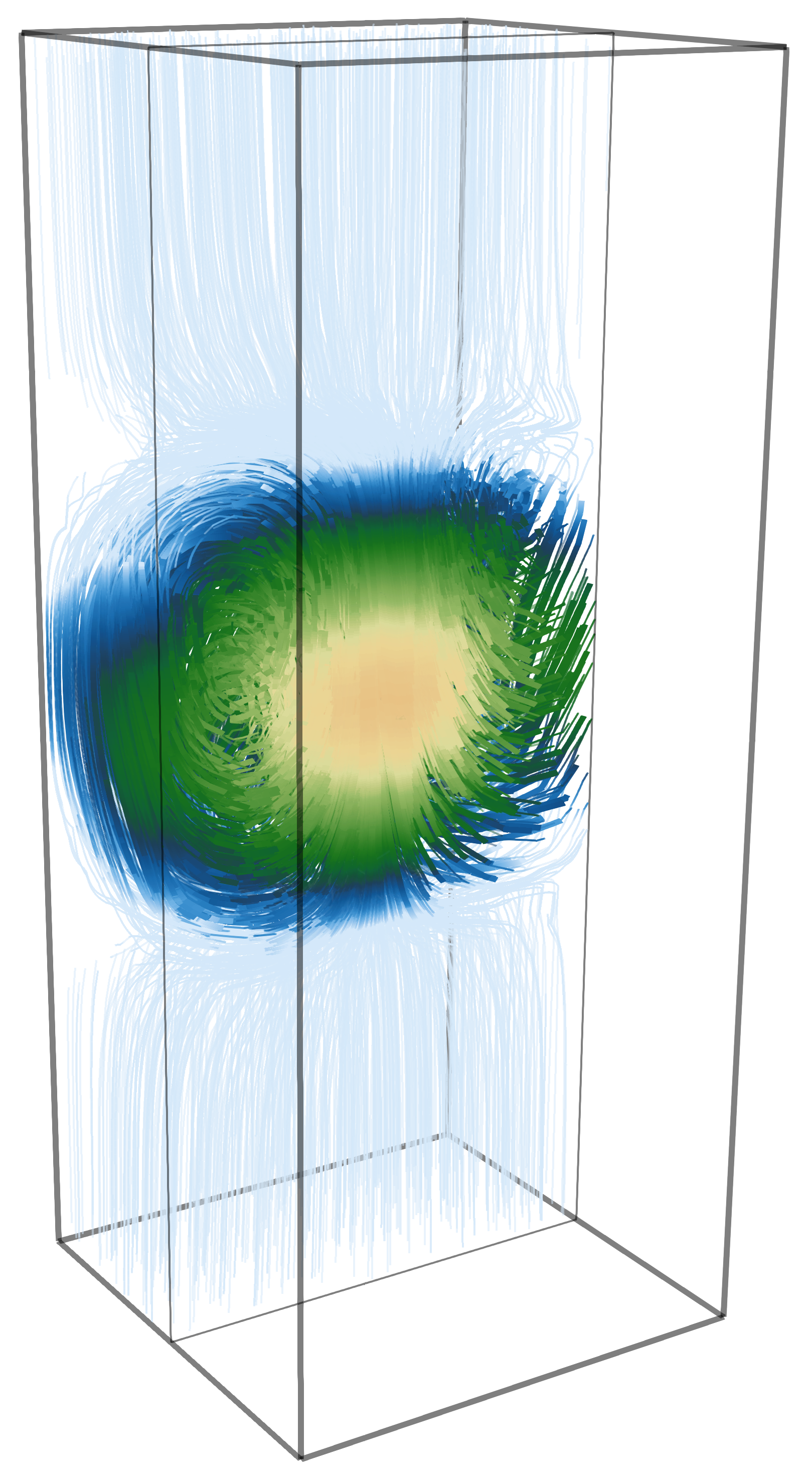}
        \caption{$t=10$}
    \end{subfigure}
    \hfill
    \begin{subfigure}{0.25\textwidth}
        \includegraphics[width=\linewidth]{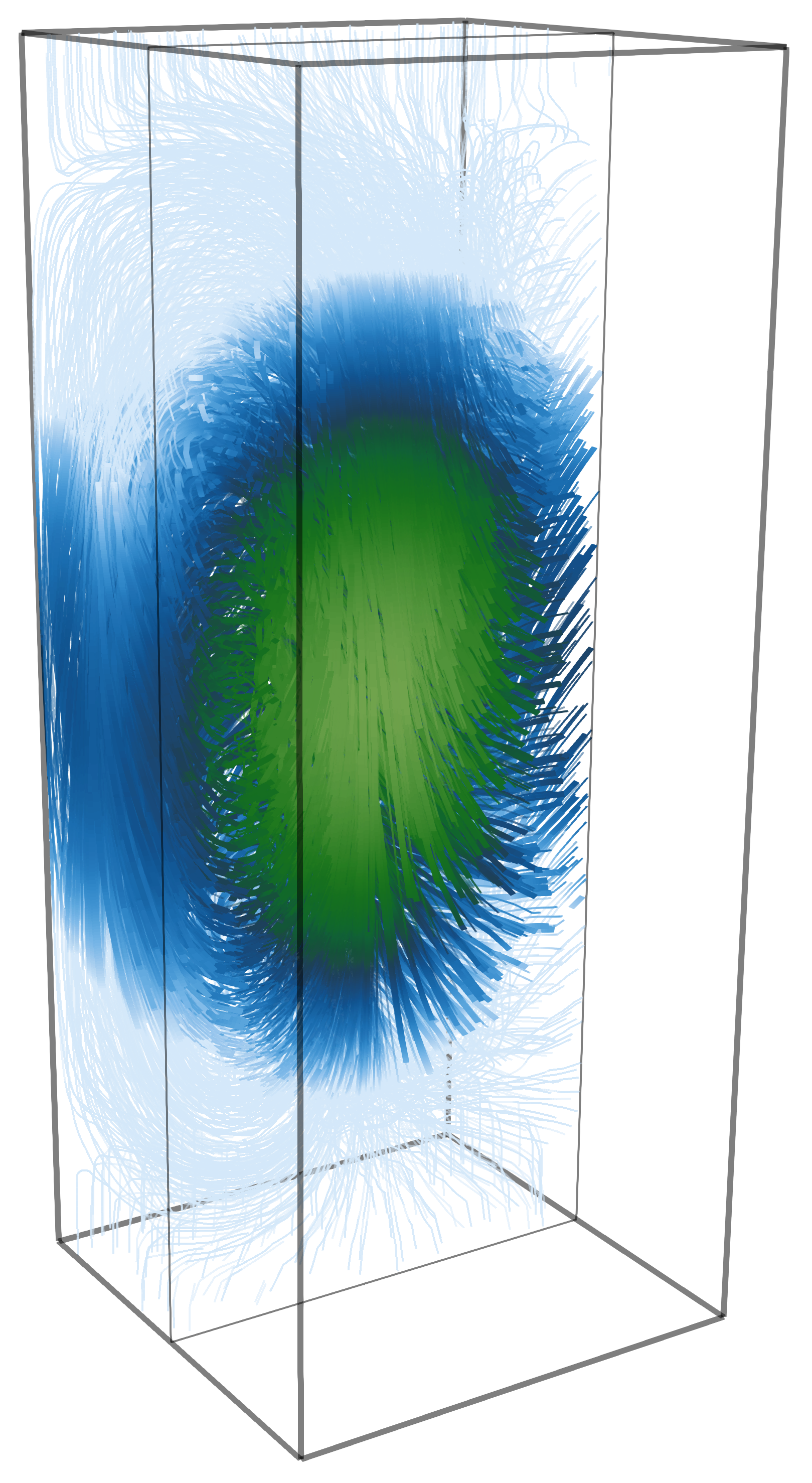}
        \caption{$t=10000$}
    \end{subfigure}
    \hfill
    \begin{subfigure}{0.14\textwidth}
        \includegraphics[width=\linewidth]{images/colorbar.png}
    \end{subfigure}
    \caption{Magnetic field lines under magnetic relaxation of the Hopf fibration, on a domain with nontrivial topology, colored by magnetic field strength $\|\bm B_h\|$.}
    \label{fig:simulation-nontrivial-domain}
\end{figure}

\section{Conclusions}\label{sec:Conclusion}
We presented a finite-element discretization \eqref{eq:structure-preserving-scheme} for the magneto-frictional system \eqref{eqn:magneto-frictional-equations}, with novel structure-preserving properties \eqref{eq:sp-properties} that are essential for numerical investigations into the Parker conjecture.
Through the conservation of helicity \eqref{eq:helicity-conservation}, the scheme preserves a discrete version \eqref{eqn:discrete-arnold-inequality} of the topological barrier provided by the Arnold inequality \eqref{eqn:arnold-inequality}, preventing the decay to spurious trivial solutions.
Extending the helicity and the Arnold inequality to certain topologically nontrivial domains, we see our scheme further retains these structures.
Numerical results confirm these structure-preserving properties.

%but standard discretization still provides reasonable results. Whereas here, the improvement is qualitative—without helicity preservation, the solutions collapse to unphysical trivial states.

The proposed method offers a promising tool for the numerical investigation of the Parker conjecture without the use of Lagrangian formulations, as well as related questions, including the formation of finite-time singularities during magnetic relaxation.
In future work, we aim to develop numerical tools to distinguish whether the steady state exhibits tangential discontinuities as suggested in the Parker conjecture. We also intend to investigate efficient preconditioners to ease scaling to larger problems. We aim further to employ adaptive timestep strategies to accelerate the magnetic relaxation process, allowing faster convergence to steady states while still maintaining the Arnold inequality's topological barrier.

Beyond solving the magneto-frictional system itself, our work serves as a striking example of the importance of structure preservation. While for many problems, such methods are known to improve the accuracy of standard discretizations, the benefit for magnetic relaxation is both quantitative and qualitative: helicity preservation is essential to ensure that numerical solutions cannot collapse to unphysical trivial states.

\section{Code availability}
The simulations in \Cref{sec:Numerical-experiment} were implemented in {Firedrake} \cite{FiredrakeUserManual} and {PETSc} \cite{petsc-user-ref};
MUMPS~\cite{amestoy2001} was used to solve the linear systems.
The code used to generate the numerical results and all {Firedrake} components have been archived on Zenodo~\cite{hehelicity2025,zenodoboris}.

\section*{Acknowledgments}
%We acknowledge many helpful discussions with 
We would like to thank
Gunnar Hornig, Yang Liu, Pablo Brubeck Martinez, Zhongmin Qian, and Chris Smiet for helpful discussions.
For the purpose of open access, the authors have applied a CC BY public copyright licence to any author accepted manuscript (AAM) arising from this submission. No new data were generated or analyzed during this work.

\appendix
\begin{comment}
\section{Proof of the Arnold inequality}\label{sec:proof-arnold-ineq}
\begin{theorem}[Arnold inequality \cite{arnold1974asymptotic}]
For a divergence-free field $\bm B$ over a cavity-free (Betti number $\beta_2 = 0$) domain $\Omega$, satisfying $\bm B \cdot \bm n = 0$ on the boundary $\partial \Omega$, we have
\begin{equation}
    \left| \int_{\Omega} \bm A \cdot \bm B \, dx \right|
    \le C \int_{\Omega} \bm B\cdot \bm B \, dx.
\end{equation}
where $C$ is a positive constant dependent only on the domain $\Omega$, and $\bm A$ is any vector potential such that $\bm B = \curl \bm A$ and $\bm A\times \bm n = 0$ on the boundary $\partial\Omega$.
\end{theorem}
\begin{proof}
The result follows from the Cauchy--Schwarz and generalized Poincar\'e inequalities respectively:
\begin{equation}
    \left| \int_{\Omega} \bm A \cdot \bm B \, dx \right|
    \le  \|\bm A\|\, \|\bm B\|
    \le  C \|\bm B\|^2.
\end{equation}
\end{proof}
\end{comment}
\section{Proof of the Arnold inequality}\label{sec:proof-arnold-ineq}
\begin{theorem}[Arnold inequality \cite{arnold1974asymptotic}]
Let $\bm B$ be a divergence-free field over a topologically trivial domain $\Omega$, satisfying $\bm B \cdot \bm n = 0$ on $\partial \Omega$. Then there exists a vector potential $\bm A$ such that $\bm B = \curl \bm A$, $\bm A \times \bm n = 0$ on $\partial \Omega$, and
\begin{equation}
    \left| \int_{\Omega} \bm A \cdot \bm B \dx \right|
    \le C \int_{\Omega} \bm B \cdot\bm B \dx,
\end{equation}
where $C>0$ depends only on $\Omega$. Moreover, the quantity $\int_{\Omega} \bm A \cdot \bm B \dx$ is independent of the particular choice of $\bm A$.
\end{theorem}
\begin{proof}
For a topologically trivial domain, one can choose $\bm A$ to be $L^2$-orthogonal to the kernel of the curl operator, ensuring that $\bm A$ satisfies a generalized Poincaré inequality
\begin{equation}
 \|\bm A\| \le C \|\curl \bm A\| = C \|\bm B\|.   
\end{equation}
Applying the Cauchy--Schwarz inequality then gives
\begin{equation}
 \left| \int_{\Omega} \bm A \cdot \bm B \dx \right|
    \le \|\bm A\|\, \|\bm B\|
    \le C \|\bm B\|^2.    
\end{equation}
If another potential $\bm A'$ satisfies $\curl \bm A' = \bm B$, then $\bm A' - \bm A$ lies in $\ker(\curl)$ and is $L^2$-orthogonal to $\bm B$, so the integral $\int_\Omega \bm A\cdot \bm B \dx$ is independent of the choice of $\bm A$.
\end{proof}

\bibliographystyle{siamplain}
\bibliography{ref}

@article{amestoy2001,
  author  = {P. R. Amestoy and I. S. Duff and J. Koster and J.-Y. L'Excellent},
  title   = {A fully asynchronous multifrontal solver using distributed dynamic scheduling},
  journal = {SIAM Journal on Matrix Analysis and Applications},
  volume  = {23},
  number  = {1},
  year    = {2001},
  pages   = {15--41}
}

@article{andrews2024enforcing,
  title   = {Enforcing conservation laws and dissipation inequalities numerically via auxiliary variables},
  author  = {Andrews, Boris D and Farrell, Patrick E},
  journal = {arXiv preprint arXiv:2407.11904},
  year    = {2024}
}

@misc{AndrewsHighorderconservativeaccurately2024,
  title         = {High-Order Conservative and Accurately Dissipative Numerical Integrators via Auxiliary Variables},
  author        = {Andrews, Boris D. and Farrell, Patrick E.},
  year          = {2024},
  number        = {arXiv:2407.11904},
  eprint        = {2407.11904},
  primaryclass  = {cs, math},
  publisher     = {arXiv},
  urldate       = {2024-09-03},
  archiveprefix = {arXiv}
}

@misc{archer2,
  doi       = {10.5281/ZENODO.14507040},
  url       = {https://zenodo.org/doi/10.5281/zenodo.14507040},
  author    = {Beckett,  George and Beech-Brandt,  Josephine and Leach,  Kieran and Payne,  Z\"{o}e and Simpson,  Alan and Smith,  Lorna and Turner,  Andy and Whiting,  Anne},
  language  = {en},
  title     = {{ARCHER2} Service Description},
  publisher = {Zenodo},
  year      = {2024},
  copyright = {Creative Commons Attribution 4.0 International}
}

@article{arnold1974asymptotic,
  title     = {The asymptotic {H}opf invariant and its applications},
  author    = {Arnold, Vladimir I},
  journal   = {Vladimir I. Arnold-Collected Works: Hydrodynamics, Bifurcation Theory, and Algebraic Geometry 1965-1972},
  pages     = {357--375},
  year      = {1974},
  publisher = {Springer}
}

@article{arnoldFiniteElementExterior2006,
  title   = {Finite Element Exterior Calculus, Homological Techniques, and Applications},
  author  = {Arnold, Douglas N. and Falk, Richard S. and Winther, Ragnar},
  year    = {2006},
  journal = {Acta Numerica},
  volume  = {15},
  pages   = {1--155},
  issn    = {0962-4929, 1474-0508},
  doi     = {10.1017/S0962492906210018},
  urldate = {2023-11-09},
  langid  = {english}
}

@article{arnoldFiniteElementExterior2010,
  title      = {Finite Element Exterior Calculus: from {{Hodge}} Theory to Numerical Stability},
  shorttitle = {Finite Element Exterior Calculus},
  author     = {Arnold, Douglas and Falk, Richard and Winther, Ragnar},
  year       = {2010},
  journal    = {Bulletin of the American Mathematical Society},
  volume     = {47},
  number     = {2},
  pages      = {281--354},
  issn       = {0273-0979, 1088-9485},
  doi        = {10.1090/S0273-0979-10-01278-4},
  urldate    = {2023-11-09},
  langid     = {english}
}

@book{ArnoldFiniteElementExterior2018,
  title     = {Finite {{Element Exterior Calculus}}},
  author    = {Arnold, Douglas N.},
  year      = {2018},
  publisher = {{Society for Industrial and Applied Mathematics}},
  address   = {Philadelphia, PA},
  doi       = {10.1137/1.9781611975543},
  urldate   = {2024-01-06},
  isbn      = {978-1-61197-553-6 978-1-61197-554-3},
  langid    = {english}
}

@book{ArnoldTopologicalMethodsHydrodynamics2021,
  title     = {Topological {{Methods}} in {{Hydrodynamics}}},
  author    = {Arnold, Vladimir I. and Khesin, Boris A.},
  year      = {2021},
  series    = {Applied {{Mathematical Sciences}}},
  volume    = {125},
  publisher = {Springer International Publishing},
  address   = {Cham},
  doi       = {10.1007/978-3-030-74278-2},
  urldate   = {2024-01-06},
  isbn      = {978-3-030-74277-5 978-3-030-74278-2},
  langid    = {english}
}

@article{beekie2022moffatt,
  title     = {On {M}offatt's magnetic relaxation equations},
  author    = {Beekie, Rajendra and Friedlander, Susan and Vicol, Vlad},
  journal   = {Communications in Mathematical Physics},
  volume    = {390},
  number    = {3},
  pages     = {1311--1339},
  year      = {2022},
  publisher = {Springer}
}

@techreport{bevir1980relaxation,
  title       = {Relaxation, flux consumption and quasi steady state pinches},
  author      = {Bevir, MK and Gray, JW},
  year        = {1980},
  institution = {Los Alamos National Laboratory},
  number      = {LA-8944-C}
}

@article{brenierTopologyPreservingDiffusionDivergenceFree2014,
  title     = {Topology-{{Preserving Diffusion}} of {{Divergence-Free Vector Fields}} and {{Magnetic Relaxation}}},
  author    = {Brenier, Yann},
  year      = {2014},
  journal   = {Communications in Mathematical Physics},
  volume    = {330},
  number    = {2},
  pages     = {757--770},
  issn      = {0010-3616, 1432-0916},
  doi       = {10.1007/s00220-014-1967-3},
  urldate   = {2024-05-22},
  copyright = {http://www.springer.com/tdm},
  langid    = {english}
}

@article{brezzi1985two,
  title     = {Two families of mixed finite elements for second order elliptic problems},
  author    = {Brezzi, Franco and Douglas, Jim and Marini, L. Donatella},
  journal   = {Numerische Mathematik},
  volume    = {47},
  pages     = {217--235},
  year      = {1985},
  publisher = {Springer}
}

@misc{candelaresiMimeticMethodsLagrangian2014a,
  title         = {Mimetic {{methods}} for {{Lagrangian relaxation}} of {{magnetic fields}}},
  author        = {Candelaresi, Simon and Pontin, David and Hornig, Gunnar},
  year          = {2014},
  number        = {arXiv:1405.0942},
  eprint        = {1405.0942},
  primaryclass  = {astro-ph, physics:physics},
  publisher     = {arXiv},
  urldate       = {2024-01-19},
  archiveprefix = {arXiv},
  langid        = {english}
}

@article{Chodura3DcodeMHD1981,
  title     = {A {{3D}} Code for {{MHD}} Equilibrium and Stability},
  author    = {Chodura, R. and Schl{\"u}ter, A.},
  year      = {1981},
  journal   = {Journal of Computational Physics},
  volume    = {41},
  number    = {1},
  pages     = {68--88},
  issn      = {00219991},
  doi       = {10.1016/0021-9991(81)90080-2},
  urldate   = {2024-08-29},
  copyright = {https://www.elsevier.com/tdm/userlicense/1.0/},
  langid    = {english}
}

@article{craig1986dynamic,
  title   = {A dynamic relaxation technique for determining the structure and stability of coronal magnetic fields},
  author  = {Craig, I. J. D. and Sneyd, A. D.},
  journal = {Astrophysical Journal, Part 1 (ISSN 0004-637X), vol. 311, Dec. 1, 1986, p. 451-459.},
  volume  = {311},
  pages   = {451--459},
  year    = {1986}
}

@article{craig2005parker,
  title     = {The {Parker} problem and the theory of coronal heating},
  author    = {Craig, I. J. D. and Sneyd, A. D.},
  journal   = {Solar Physics},
  volume    = {232},
  pages     = {41--62},
  year      = {2005},
  publisher = {Springer}
}

@article{craig2014current,
  title     = {Current singularities in line-tied three-dimensional magnetic fields},
  author    = {Craig, I. J. D. and Pontin, DI},
  journal   = {The Astrophysical Journal},
  volume    = {788},
  number    = {2},
  pages     = {177},
  year      = {2014},
  publisher = {IOP Publishing}
}

@article{enciso2025obstructions,
  title     = {Obstructions to topological relaxation for generic magnetic fields},
  author    = {Enciso, Alberto and Peralta-Salas, Daniel},
  journal   = {Archive for Rational Mechanics and Analysis},
  volume    = {249},
  number    = {1},
  pages     = {1--28},
  year      = {2025},
  publisher = {Springer}
}

@manual{FiredrakeUserManual,
  title        = {Firedrake User Manual},
  author       = {David A. Ham and Paul H. J. Kelly and Lawrence Mitchell and Colin J. Cotter and Robert C. Kirby and Koki Sagiyama and Nacime Bouziani and Sophia Vorderwuelbecke and Thomas J. Gregory and Jack Betteridge and Daniel R. Shapero and Reuben W. Nixon-Hill and Connor J. Ward and Patrick E. Farrell and Pablo D. Brubeck and India Marsden and Thomas H. Gibson and Miklós Homolya and Tianjiao Sun and Andrew T. T. McRae and Fabio Luporini and Alastair Gregory and Michael Lange and Simon W. Funke and Florian Rathgeber and Gheorghe-Teodor Bercea and Graham R. Markall},
  organization = {Imperial College London and University of Oxford and Baylor University and University of Washington},
  edition      = {First edition},
  year         = {2023},
  month        = {5},
  doi          = {10.25561/104839}
}

@article{gawlikFiniteElementMethod2022,
  title   = {A Finite Element Method for {{MHD}} That Preserves Energy, Cross-Helicity, Magnetic Helicity, Incompressibility, and Div {{B}} = 0},
  author  = {Gawlik, Evan S. and {Gay-Balmaz}, Fran{\c c}ois},
  year    = {2022},
  journal = {Journal of Computational Physics},
  volume  = {450},
  pages   = {110847},
  issn    = {00219991},
  doi     = {10.1016/j.jcp.2021.110847},
  urldate = {2024-01-25},
  langid  = {english}
}

@article{GuzmanConformingdivergencefreeStokes2014,
  title   = {Conforming and Divergence-Free {{Stokes}} Elements in Three Dimensions},
  author  = {Guzm\'an, J. and Neilan, M.},
  year    = {2014},
  journal = {IMA Journal of Numerical Analysis},
  volume  = {34},
  number  = {4},
  pages   = {1489--1508},
  issn    = {0272-4979, 1464-3642},
  doi     = {10.1093/imanum/drt053},
  urldate = {2024-05-10},
  langid  = {english}
}

@article{hairer2006geometric,
  title   = {Geometric numerical integration},
  author  = {Hairer, Ernst and Hochbruck, Marlis and Iserles, Arieh and Lubich, Christian},
  journal = {Oberwolfach Reports},
  volume  = {3},
  number  = {1},
  pages   = {805--882},
  year    = {2006}
}

@misc{hornigUniversalMagneticHelicity2006,
  title         = {A {{Universal Magnetic Helicity Integral}}},
  author        = {Hornig, Gunnar},
  year          = {2006},
  number        = {arXiv:astro-ph/0606694},
  eprint        = {astro-ph/0606694},
  publisher     = {arXiv},
  urldate       = {2024-03-04},
  archiveprefix = {arXiv},
  langid        = {english}
}

@article{huHelicityconservativeFiniteElement2021,
  title   = {Helicity-Conservative Finite Element Discretization for Incompressible {{MHD}} Systems},
  author  = {Hu, Kaibo and Lee, Young-Ju and Xu, Jinchao},
  year    = {2021},
  journal = {Journal of Computational Physics},
  volume  = {436},
  pages   = {110284},
  issn    = {00219991},
  doi     = {10.1016/j.jcp.2021.110284},
  urldate = {2023-12-24},
  langid  = {english}
}

@article{huStableFiniteElement2017,
  title     = {Stable finite element methods preserving {$\nabla\cdot B=0$} exactly for MHD models},
  author    = {Hu, Kaibo and Ma, Yicong and Xu, Jinchao},
  journal   = {Numerische Mathematik},
  volume    = {135},
  number    = {2},
  pages     = {371--396},
  year      = {2017},
  publisher = {Springer}
}

@article{huStructurepreservingFiniteElement2018,
  title   = {Structure-Preserving Finite Element Methods for Stationary {{MHD}} Models},
  author  = {Hu, Kaibo and Xu, Jinchao},
  year    = {2018},
  journal = {Mathematics of Computation},
  volume  = {88},
  number  = {316},
  pages   = {553--581},
  issn    = {0025-5718, 1088-6842},
  doi     = {10.1090/mcom/3341},
  urldate = {2023-11-07},
  langid  = {english}
}

@article{LaakmannStructurepreservinghelicityconservingfinite2023,
  title   = {Structure-Preserving and Helicity-Conserving Finite Element Approximations and Preconditioning for the {{Hall MHD}} Equations},
  author  = {Laakmann, Fabian and Hu, Kaibo and Farrell, Patrick E.},
  year    = {2023},
  journal = {Journal of Computational Physics},
  volume  = {492},
  pages   = {112410},
  issn    = {00219991},
  doi     = {10.1016/j.jcp.2023.112410},
  urldate = {2024-01-07},
  langid  = {english}
}

@article{liu2025obtaining,
  title={Obtaining Pseudoinverse Solutions with MINRES},
  author={Liu, Yang and Milzarek, Andre and Roosta, Fred},
  journal={SIAM Journal on Matrix Analysis and Applications},
  volume={46},
  number={3},
  pages={1887--1916},
  year={2025},
  publisher={SIAM}
}

@article{longbottom1998magnetic,
  title     = {Magnetic flux braiding: force-free equilibria and current sheets},
  author    = {Longbottom, AW and Rickard, GJ and Craig, I. J. D. and Sneyd, A. D.},
  journal   = {The Astrophysical Journal},
  volume    = {500},
  number    = {1},
  pages     = {471},
  year      = {1998},
  publisher = {IOP Publishing}
}

@article{mactaggart2019magnetic,
  title     = {Magnetic helicity in multiply connected domains},
  author    = {MacTaggart, David and Valli, Alberto},
  journal   = {Journal of Plasma Physics},
  volume    = {85},
  number    = {5},
  pages     = {775850501},
  year      = {2019},
  publisher = {Cambridge University Press}
}

@article{MaRobustpreconditionersincompressible2016,
  title   = {Robust Preconditioners for Incompressible {{MHD}} Models},
  author  = {Ma, Yicong and Hu, Kaibo and Hu, Xiaozhe and Xu, Jinchao},
  year    = {2016},
  journal = {Journal of Computational Physics},
  volume  = {316},
  pages   = {721--746},
  issn    = {00219991},
  doi     = {10.1016/j.jcp.2016.04.019},
  urldate = {2023-11-07},
  langid  = {english}
}

@article{moffatt1981some,
  title     = {Some developments in the theory of turbulence},
  author    = {Moffatt, HK},
  journal   = {Journal of Fluid Mechanics},
  volume    = {106},
  pages     = {27--47},
  year      = {1981},
  publisher = {Cambridge University Press}
}

@article{nedelec1-0,
  author  = {N\'ed\'elec, Jean-Claude},
  title   = {Mixed finite elements in \(\mathbb{R}^3\)},
  journal = {Numerische Mathematik},
  volume  = {35},
  number  = {3},
  year    = {1980},
  doi     = {10.1007/BF01396415},
  pages   = {{315--341}}
}

@article{parkerTopologicalDissipationSmallScale1972,
  title   = {Topological {{dissipation}} and the {{small-scale fields}} in {{turbulent gases}}},
  author  = {Parker, E. N.},
  year    = {1972},
  journal = {The Astrophysical Journal},
  volume  = {174},
  pages   = {499},
  issn    = {0004-637X, 1538-4357},
  doi     = {10.1086/151512},
  urldate = {2024-05-30},
  langid  = {english}
}

@techreport{petsc-user-ref,
  author      = {Satish Balay and Shrirang Abhyankar and Mark~F. Adams and Steven Benson and Jed
                 Brown and Peter Brune and Kris Buschelman and Emil Constantinescu and Lisandro
                 Dalcin and Alp Dener and Victor Eijkhout and Jacob Faibussowitsch and William~D.
                 Gropp and V\'{a}clav Hapla and Tobin Isaac and Pierre Jolivet and Dmitry Karpeev
                 and Dinesh Kaushik and Matthew~G. Knepley and Fande Kong and Scott Kruger and
                 Dave~A. May and Lois Curfman McInnes and Richard Tran Mills and Lawrence Mitchell
                 and Todd Munson and Jose~E. Roman and Karl Rupp and Patrick Sanan and Jason Sarich
                 and Barry~F. Smith and Hansol Suh and Stefano Zampini and Hong Zhang and Hong Zhang
                 and Junchao Zhang},
  title       = {{PETSc/TAO} Users Manual},
  institution = {Argonne National Laboratory},
  number      = {ANL-21/39 - Revision 3.22},
  doi         = {10.2172/2205494},
  year        = {2024}
}

@article{pontinParkerProblemExistence2020,
  title      = {The {{Parker}} Problem: existence of Smooth Force-Free Fields and Coronal Heating},
  shorttitle = {The {{Parker}} Problem},
  author     = {Pontin, David I. and Hornig, Gunnar},
  year       = {2020},
  journal    = {Living Reviews in Solar Physics},
  volume     = {17},
  number     = {1},
  pages      = {5},
  issn       = {2367-3648, 1614-4961},
  doi        = {10.1007/s41116-020-00026-5},
  urldate    = {2024-01-09},
  langid     = {english}
}

@inproceedings{raviart2006mixed,
  title        = {A mixed finite element method for 2-nd order elliptic problems},
  author       = {Raviart, Pierre-Arnaud and Thomas, Jean-Marie},
  booktitle    = {Mathematical Aspects of Finite Element Methods: Proceedings of the Conference Held in Rome, December 10--12, 1975},
  pages        = {292--315},
  year         = {2006},
  organization = {Springer}
}

@article{schotzau2004mixed,
  title     = {Mixed finite element methods for stationary incompressible magneto--hydrodynamics},
  author    = {Sch{\"o}tzau, Dominik},
  journal   = {Numerische Mathematik},
  volume    = {96},
  number    = {4},
  pages     = {771--800},
  year      = {2004},
  publisher = {Springer}
}

@article{smietIdealRelaxationHopf2017,
  title         = {Ideal {{relaxation}} of the {{Hopf fibration}}},
  author        = {Smiet, Christopher Berg and Candelaresi, Simon and Bouwmeester, Dirk},
  year          = {2017},
  journal       = {Physics of Plasmas},
  volume        = {24},
  number        = {7},
  eprint        = {1610.04719},
  primaryclass  = {physics},
  pages         = {072110},
  issn          = {1070-664X, 1089-7674},
  doi           = {10.1063/1.4990076},
  urldate       = {2024-03-30},
  archiveprefix = {arXiv},
  langid        = {english}
}

@article{wilmot2009magneticparallel,
  title     = {Magnetic braiding and parallel electric fields},
  author    = {Wilmot-Smith, AL and Hornig, G and Pontin, DI},
  journal   = {The Astrophysical Journal},
  volume    = {696},
  number    = {2},
  pages     = {1339},
  year      = {2009},
  publisher = {IOP Publishing}
}

@article{wilmot2009magneticquasi,
  title     = {Magnetic braiding and quasi-separatrix layers},
  author    = {Wilmot-Smith, AL and Hornig, G and Pontin, DI},
  journal   = {The Astrophysical Journal},
  volume    = {704},
  number    = {2},
  pages     = {1288},
  year      = {2009},
  publisher = {IOP Publishing}
}

@article{yeates2022limitations,
  title     = {On the limitations of magneto-frictional relaxation},
  author    = {Yeates, AR},
  journal   = {Geophysical \& Astrophysical Fluid Dynamics},
  volume    = {116},
  number    = {4},
  pages     = {305--320},
  year      = {2022},
  publisher = {Taylor \& Francis}
}

@misc{zenodoboris,
  author = {Andrews, B. D. and Farrell, P. E.},
  title  = {Software used in `Enforcing conservation laws and dissipation inequalities numerically via auxiliary variables'},
  year   = {2025},
  doi    = {10.5281/zenodo.15302724}
}

@article{zhou2014variational,
  title     = {Variational integration for ideal magnetohydrodynamics with built-in advection equations},
  author    = {Zhou, Yao and Qin, Hong and Burby, Joshua W and Bhattacharjee, Amitava},
  journal   = {Physics of Plasmas},
  volume    = {21},
  number    = {10},
  year      = {2014},
  publisher = {AIP Publishing}
}

@article{zhou2016formation,
  title     = {Formation of current singularity in a topologically constrained plasma},
  author    = {Zhou, Yao and Huang, Yi-Min and Qin, Hong and Bhattacharjee, A},
  journal   = {Physical Review E},
  volume    = {93},
  number    = {2},
  pages     = {023205},
  year      = {2016},
  publisher = {APS}
}

@article{zhou2017constructing,
  title     = {Constructing current singularity in a {3D} line-tied plasma},
  author    = {Zhou, Yao and Huang, Yi-Min and Qin, Hong and Bhattacharjee, Amitava},
  journal   = {The Astrophysical Journal},
  volume    = {852},
  number    = {1},
  pages     = {3},
  year      = {2017},
  publisher = {IOP Publishing}
}

@article{zweibel1987formation,
  title   = {The formation of current sheets in the solar atmosphere},
  author  = {Zweibel, Ellen G and Li, He-Sheng},
  journal = {The Astrophysical Journal},
  volume  = {312},
  pages   = {423--430},
  year    = {1987}
}

@article{pontin2009lagrangian,
  title={Lagrangian relaxation schemes for calculating force-free magnetic fields, and their limitations},
  author={Pontin, DI and Hornig, G and Wilmot-Smith, AL and Craig, Ian JD},
  journal={The Astrophysical Journal},
  volume={700},
  number={2},
  pages={1449},
  year={2009},
  publisher={IOP Publishing}
}

@misc{hehelicity2025,
 key   = {zenodo/Zenodo-20250811.3},
 title = {{Software used in `Helicity-preserving finite element discretization for magnetic rel
axation'}},
 year  = {2025},
 month = {aug},
 doi   = {10.5281/zenodo.16797562},
 url   = {https://doi.org/10.5281/zenodo.16797562},
}
\end{document}